\documentclass[11pt,a4paper]{article}
\usepackage{ifpdf}
\usepackage[utf8]{inputenc}
\usepackage[T1]{fontenc}
\usepackage{graphicx}
\usepackage{color}
\usepackage{textcomp}
\usepackage{amsmath,amssymb,amsthm}
\usepackage{geometry}
\usepackage[all]{xy}

\usepackage[usenames,dvipsnames]{pstricks}
\usepackage{epsfig}
\usepackage{pst-grad} 
\usepackage{pst-plot} 
\title{Constraint equations for $3+1$ vacuum Einstein equations with a translational space-like Killing field in the asymptotically flat case II}
\author{Cécile Huneau}

\ifpdf
\fi

\newtheorem{thm}{Theorem}[section]
\newtheorem{prp}[thm]{Proposition}
\newtheorem{cor}[thm]{Corollary}
\newtheorem{lm}[thm]{Lemma}
\newtheorem{df}[thm]{Definition}
\newtheorem{rk}[thm]{Remark}

\newcommand{\m}[1]{\mathbb{#1}}
\newcommand{\ep}{\varepsilon}
\newcommand{\q}[1]{\mathcal{#1}}

\newcommand{\wht}[1]{\widetilde{#1}}
\newcommand{\gra}[1]{\mathbf{#1}}
\newcommand{\grat}[1]{\mathbf{\widetilde{#1}}}

\begin{document}

\maketitle

\begin{abstract}
 We solve the Einstein constraint equations for a $3+1$ dimensional vacuum space-time with a space-like translational Killing field in the asymptotically flat case.
The presence of a space-like translational Killing field allows for a reduction of the $3+1$ dimensional problem to a $2+1$ dimensional one.
The aim of this paper is to go further in the asymptotic expansion of the solutions than in \cite{moi}. In particular the expansion we construct involves quantities which are the 2-dimensional equivalent of the global charges.
\end{abstract}

\section{Introduction}
Einstein equations can be formulated as a Cauchy problem whose initial data must satisfy compatibility conditions known as the constraint equations.
In this paper, we will consider the constraint equations for the vacuum Einstein equations, in the particular case where the space-time possesses a 
space-like translational Killing field. It allows for a reduction of the $3+1$ dimensional
problem to a $2+1$ dimensional one. This symmetry has been studied by Choquet-Bruhat and Moncrief in \cite{choquet} (see also \cite{livrecb})
in the case of a
space-time of the form $\Sigma \times \mathbb{S}^1  \times \mathbb{R}$, where $\Sigma$ is a compact two dimensional manifold of genus $G\geq 2$, and $\mathbb{R}$
is the time axis, with a space-time metric independent of the $\mathbb{S}^1$ coordinate. 
They prove the existence of global solutions corresponding to perturbation of particular expanding initial data.

In this paper we consider a space-time of the form $\mathbb{R}^2 \times \mathbb{R}_{x_3} \times \mathbb{R}_{t}$, symmetric with respect to the third coordinate.
Minkowski space-time is a particular solution of vacuum Einstein equations which exhibits this symmetry. Since the celebrated work of Christodoulou and Klainerman (see \cite{ck}), we know that Minkowski space-time is stable, that is to say asymptotically flat perturbations of the trivial initial data lead to global solutions converging to Minkowski space-time.
It is an interesting problem to ask whether
the stability also holds in the setting of perturbations of Minkowski space-time with a space-like translational Killing field. Let's note that it is not included in the work of Christodoulou and Klainerman. However, it is crucial, before considering this problem, 
to ensure the existence of compatible initial data. In \cite{moi}, we proved the existence of solutions to the constraint equations. The purpose of this paper is to go further in the asymptotic development of the solutions to the constraint equations. The solutions we construct in this paper are actually the one used in \cite{quasistab}
to prove the stability in exponential time of Minkowski space-time with a space-like translational Killing field.

In the compact case, if one looks for solutions with constant mean curvature, as it is done in \cite{choquet}, the issue of solving the constraint equations is straightforward. 
Every metric on a compact manifold of genus $G\geq 2$ is conformal
to a metric of scalar curvature $-1$. As a consequence, it is possible to decouple the system into elliptic
scalar equations of the form $\Delta u = f(x,u)$ with $\partial_u f >0$, for which existence results  are standard (see for example chapter $14$ in \cite{tay}).

The asymptotically flat case is more challenging. First, the definition of an asymptotically flat manifold is not so clear in two dimension. 
In \cite{beck}, \cite{asht}, \cite{chru} radial solutions of the $2+1$ dimensional problem with an angle at space-like infinity are constructed.
In particular, these solutions do not tend to the Euclidean metric at space-like infinity. Moreover, the behavior of the Laplace operator on $\m R^2$ makes the issue of finding solutions to the constraint equations
more intricate.

\subsection{Reduction of the Einstein equations}
Before discussing the constraint equations, we first briefly recall the form of the Einstein equations in the presence of a space-like translational 
Killing field. We follow here the exposition
in \cite{livrecb}.
A metric $^{(4)}\mathbf{g}$ on  $\mathbb{R}^2 \times \m R  \times \mathbb{R}$ admitting $\partial_3$ as a Killing field can be written
$$^{(4)}\mathbf{g} = \grat g + e^{2\gamma}(dx^3 +A_\alpha dx^\alpha)^2, $$
where $\grat g$ is a Lorentzian metric on $\mathbb{R}^{1+2}$, $\gamma$ is a scalar function on $\mathbb{R}^{1+2}$, $A$ is a $1$-form on
$\mathbb{R}^{1+2}$ and $x^\alpha$, $\alpha=0,1,2$, are the coordinates on $\mathbb{R}^{1+2}$. Since $\partial_3$ is a Killing field, $\mathbf{g}$, $\gamma$ and $A$ do not depend on $x^3$.
We set $F=dA$, where $d$ is the exterior differential. $F$ is then a $2$-form. Let also $^{(4)}\mathbf{R}_{\mu \nu}$ denote the Ricci tensor associated to $^{(4)}\mathbf{g}$.
$\grat R_{\alpha \beta}$ and $\grat D$ are respectively the Ricci tensor and the covariant derivative associated to $\grat g$.

With this metric, the vacuum Einstein equations
$$^{(4)}\mathbf{R}_{\mu \nu} = 0, \; \mu, \nu = 0,1,2,3$$
can be written in the basis $(dx^\alpha, dx^3+A_\alpha dx^\alpha)$  (see \cite{livrecb} appendix VII)
\begin{align}
\label{chap2:rab}0= ^{(4)}\gra R_{\alpha \beta} &= \grat R_{\alpha \beta}-\frac{1}{2}e^{2\gamma}{F_\alpha}^\lambda F_{\beta \lambda}
 -\grat D_\alpha \partial_\beta \gamma -\partial_\alpha \gamma \partial_\beta \gamma,\\
\label{chap2:ra3} 0=^{(4)}\gra R_{\alpha 3} &= \frac{1}{2}e^{-\gamma} \grat D_\beta (e^{3\gamma}{F_\alpha}^\beta),\\
\label{chap2:r33} 0=^{(4)}\gra R_{33}&= -e^{-2\gamma} \left( -\frac{1}{4}e^{2\gamma}F_{\alpha \beta}F^{\alpha \beta} + \grat g^{\alpha \beta}\partial_\alpha \gamma \partial_\beta \gamma
 + \grat g^{\alpha \beta}\grat D_\alpha \partial_\beta \gamma\right).
\end{align} 

The equation \eqref{chap2:ra3} is equivalent to 
$$d(\ast e^{3\gamma}F)=0$$
where $\ast e^{3\gamma}F$ is the adjoint one form associated to $e^{3\gamma}F$.
This is equivalent, on $\m R^{1+2}$, to the existence of a potential $ \omega$ such that
$$\ast e^{3\gamma}F=d\omega.$$
Since $F$ is a closed $2$-form, we have $dF=0$. By doing the conformal change of metric $\grat g = e^{-2\gamma}\gra g$, 
this equation, together with the equations \eqref{chap2:rab} and \eqref{chap2:r33}, yield the following system,

\begin{align}
\label{chap2:wavebis} &\Box_{\gra g} \omega -4 \partial^\alpha \gamma \partial_\alpha \omega=0,\\
 \label{chap2:wave}& \Box_\mathbf{g} \gamma +\frac{1}{2}e^{-4\gamma} \partial^\alpha \omega \partial_\alpha \omega= 0, \\
\label{chap2:einst2} &\mathbf{R}_{\alpha \beta} = 2\partial_\alpha \gamma\partial_\beta \gamma +\frac{1}{2}e^{-4\gamma}\partial_\alpha \omega \partial_\beta \omega, \; \alpha,\beta = 0,1,2,
\end{align}
where $\Box_\mathbf{g}$ is the d'Alembertian\footnote{$\Box_\mathbf{g}$ is the
Lorentzian equivalent of the Laplace-Beltrami operator in Riemannian geometry. In a coordinate system,
we have $\Box_\mathbf{g} u = \frac{1}{\sqrt{|\mathbf{g}|}}\partial_\alpha(\mathbf{g}^{\alpha \beta}\sqrt{|\mathbf{g}|}\partial_\beta u)$.} in the metric $\mathbf{g}$
 and  $\mathbf{R}_{\alpha \beta}$ is the Ricci tensor associated to $\mathbf{g}$.
We introduce the following notation
\begin{equation}\label{chap2:defubis}
 u \equiv (\gamma,\omega),
\end{equation}
together with the scalar product
\begin{equation}\label{chap2:defu}
\partial_\alpha u. \partial_\beta u=2\partial_\alpha \gamma \partial_\beta \gamma + \frac{1}{2}e^{-4\gamma}\partial_\alpha \omega \partial_\beta \omega.
\end{equation}

We consider the Cauchy problem for the equations \eqref{chap2:wavebis}, \eqref{chap2:wave} and \eqref{chap2:einst2}.
As it is in the case for the $3+1$ Einstein equation, the initial data for \eqref{chap2:wavebis}, \eqref{chap2:wave} and \eqref{chap2:einst2} cannot be prescribed arbitrarily. They have to satisfy constraint equations.

\subsection{Constraint equations}
 We can write the metric $\mathbf{g}$ under the form
\begin{equation}\label{chap2:metrique}
\mathbf{g} = -N^2(dt)^2 +g_{ij}(dx^i +\beta^i dt)(dx^j + \beta^j dt),
\end{equation}
where the scalar function $N$ is called the lapse, the vector field $\beta$ is called the shift and $g$ is a Riemannian metric on $\m R^2$. 

We consider the initial space-like surface $\m R ^2 = \{t=0\}$. Let $T$ be the unit normal to $\m R ^2 = \{t=0\}$. We set
$$e_0=NT=\partial_t-\beta^j \partial_j.$$
We will use the notation
$$\partial_0=\mathcal{L}_{e_0}=\partial_t - \mathcal{L}_{\beta},$$
where $\mathcal{L}$ is the Lie derivative. With this notation, we have the following 
expression for the second fundamental form of $\m R^2$
$$K_{ij}=-\frac{1}{2N}\partial_0 g_{ij}.$$
We will use the notation 
$$\tau=g^{ij}K_{ij}$$
for the mean curvature. We also introduce the Einstein tensor
$$\mathbf{G}_{\alpha \beta}= \mathbf{R}_{\alpha \beta} - \frac{1}{2}\mathbf{R}\mathbf{g}_{\alpha \beta},$$
where $\mathbf{R}$ is the scalar curvature $\mathbf{R} = \mathbf{g}^{\alpha \beta}\mathbf{R}_{\alpha \beta}$. 
The constraint equations are given by
\begin{align}
 \label{chap2:contrmom} \mathbf{G}_{0j} &\equiv N(\partial_j \tau - D^i K_{ij})=\partial_0 u. \partial_j u, \; j=1,2,\\
\label{chap2:contrham} \mathbf{G}_{00} & \equiv \frac{N^2}{2}(R-|K|^2+ \tau^2)= \partial_0 u. \partial_0 u - \frac{1}{2}\mathbf{g}_{00} \mathbf{g}^{\alpha \beta}\partial_\alpha u \partial_\beta u,
\end{align}
where $D$ and $R$ are respectively the covariant derivative and the scalar curvature associated to $g$ (see \cite{livrecb} chapter VI for a derivation of \eqref{chap2:contrmom} and \eqref{chap2:contrham}).
Equation \eqref{chap2:contrmom} is called the momentum constraint and \eqref{chap2:contrham} is called the Hamiltonian constraint. If we came back to the $3+1$ problem, there should be four constraint equations. However, since the fourth would be obtained by taking $\alpha=0$ in \eqref{chap2:ra3}, it is trivially satisfied if we set $\ast e^{3\gamma}F=d\omega$.

We will look for $g$ of the form $g= e^{2\lambda}\delta$ where $\delta$ is the Euclidean metric on $\mathbb{R}^2$. There is no loss of generality since, up to a diffeomorphism, all
metrics on $\mathbb{R}^2$ are conformal to the Euclidean metric. We introduce the traceless part of $K$, 
$$H_{ij} = K_{ij} - \frac{1}{2}\tau g_{ij},$$
and following \cite{choquet} we introduce the quantity
$$\dot{u} = \frac{e^{2\lambda}}{N}\partial_0 u.$$
Then the equations \eqref{chap2:contrmom} and \eqref{chap2:contrham} take the form
\begin{align}
 \label{chap2:mombis}&\partial_i H_{ij} = -\dot{u}.\partial_j u + \frac{1}{2} e^{2\lambda} \partial_j \tau,\\
\label{chap2:hambis}& \Delta \lambda + e^{-2\lambda}\left(\frac{1}{2}\dot{u}^2 +\frac{1}{2}|H|^2\right)-e^{2\lambda}\frac{\tau^2}{4} + \frac{1}{2}|\nabla u|^2 = 0,
\end{align}
where here and in the remaining of the paper, we use the convention for the Laplace operator
$$\Delta=\partial^2_1+\partial^2_2.$$
\\

The aim of this paper is to solve the coupled system of nonlinear elliptic equations \eqref{chap2:mombis} and \eqref{chap2:hambis} on $\m R^2$ in the small
data case, that is to say when $\dot{u}$ and $\nabla u$ are small. A similar system can be obtained when studying the constraint equations in three
dimensions by using the conformal method, introduced by Lichnerowicz \cite{lich} and Choquet-Bruhat and York \cite{york}.
In the constant mean curvature (CMC) case, that is to say when one sets $\tau=0$,
the constraint equations decouple and
the main difficulty that remains is the study
of the scalar equation \eqref{chap2:hambis}, also called the Lichnerowicz equation\footnote{The resolution of this equation is closely linked to the Yamabe problem}. 
The  CMC solutions have been studied in \cite{york} and \cite{isem} for the compact case, and in \cite{cantor} for the asymptotically flat case.
There have been also some results concerning the coupled constraint equations, i.e. without setting $\tau$ constant
The near CMC solutions in the asymptotically flat case have been studied in \cite{yorkcb}.
The compact case has been studied in
\cite{holst}, \cite{maxwell} and \cite{humbert}. See also \cite{bartnik} for a review of these results.

As in \cite{moi}, the solutions we construct in this paper are of the form
$$\lambda= -\alpha\ln(r)+o(1).$$
As shown by the analysis in \cite{moi}, this logarithmic growth does not contradict asymptotic flatness, but actually corresponds to the deficit angle present in \cite{asht}.

We will do the following rescaling to avoid the $e^{2\lambda}$
and $e^{-2\lambda}$ factors
$$\breve{u}=e^{-\lambda}\dot{u}, \quad \breve{H}=e^{-\lambda}H, \quad \breve{\tau} = e^{\lambda}\tau.$$
Then the equations \eqref{chap2:mombis} and \eqref{chap2:hambis} become
\begin{align*}
&\partial_i \breve{H}_{ij} +\breve{H}_{ij}\partial_i \lambda  = -\breve{u}.\partial_j u + \frac{1}{2} \partial_j \breve{\tau}-\frac{1}{2} \breve{\tau}\partial_j \lambda,\\
& \Delta \lambda + \frac{1}{2}\breve{u}^2 + \frac{1}{2}|\nabla u|^2+\frac{1}{2}|\breve{H}|^2-\frac{\breve{\tau}^2}{4} = 0.
\end{align*}
To lighten the notations, we will omit the $\;\breve{}\;$ in the rest of the paper. We consider therefore the system
\begin{equation}\label{chap2:sys}
\left\{
\begin{array}{l}
\partial_i H_{ij} +H_{ij}\partial_i \lambda  = -\dot{u}.\partial_j u + \frac{1}{2} \partial_j \tau-\frac{1}{2} \tau\partial_j \lambda,\\
\Delta \lambda + \frac{1}{2}\dot{u}^2 + \frac{1}{2}|\nabla u|^2+\frac{1}{2}|H|^2-\frac{\tau^2}{4} = 0. \\
\end{array} \right.
\end{equation}

Before stating the main result, we recall several properties of weighted Sobolev spaces.

\section{Preliminaries}

\subsection{Weighted Sobolev spaces}
In the rest of the paper, $\chi(r)$ denotes a smooth non negative function such that 
$$0 \leq \chi \leq 1, \quad  \chi(r) =0 \; \text{for}\; r \leq 1, \quad \chi(r)=1 \; \text{for} \; r \geq 2.$$
We will also note $f \lesssim h $ when there exists a universal constant $C$ such that $f\leq Ch$.

\begin{df} Let  $m\in \m N$ and $\delta \in \mathbb{R}$. The weighted Sobolev space $H^m_\delta(\mathbb{R}^n)$ is the completion of $C^\infty_0$ for the norm 
 $$\|u\|_{H^m_\delta}=\sum_{|\beta|\leq m}\|(1+|x|^2)^{\frac{\delta +|\beta|}{2}}D^\beta u\|_{L^2}.$$
The weighted Hölder space $C^m_{\delta}$ is the complete space of $m$-times continuously differentiable functions with norm 
$$\|u\|_{C^m_{\delta}}=\sum_{|\beta|\leq m}\|(1+|x|^2)^{\frac{\delta +|\beta|}{2}}D^\beta u\|_{L^\infty}.$$
Let $0<\alpha<1$. The Hölder space $C^{m+\alpha}_\delta$ is the the complete space of $m$-times continuously differentiable functions with norm 
$$\|u\|_{C^{m+\alpha}_{\delta}}=\|u\|_{C^m_\delta} + \sup_{x \neq y, \; |x-y|\leq 1} \frac{|\partial^m u(x)-\partial^m u(y)|(1+|x|^2)^\frac{\delta}{2}}{|x-y|^\alpha}.$$
\end{df}

The following lemma is an immediate consequence of the definition.

\begin{lm}\label{chap2:der} Let $m \geq 1$ and $\delta \in \mathbb{R}$. Then $u \in H^m_\delta$ implies $\partial_j u \in H^{m-1}_{\delta+1}$ for $j=1,..,n$.
\end{lm}
We first recall the Sobolev embedding with weights (see for example \cite{livrecb}, Appendix I). In the rest of this section, we assume $n=2$.

\begin{prp}\label{chap2:holder} Let $s,m \in \m N$. We assume $s >1$. Let $\beta \leq \delta +1$ and $0<\alpha<min(1,s-1)$. Then, we have the continuous embedding
$$H^{s+m}_{\delta}\subset C^{m+\alpha}_{\beta}.$$
\end{prp}

We will also need a product rule.

\begin{prp}\label{chap2:produit}
 Let $s,s_1,s_2 \in \m N$. We assume $s\leq \min(s_1,s_2)$ and $s<s_1+s_2-1$. Let $\delta<\delta_1+\delta_2+1$. Then $\forall (u,v) \in H^{s_1}_{\delta_1}\times H^{s_2}_{\delta_2}$,
$$\|uv\|_{H^s_\delta} \lesssim \|u\|_{H^{s_1}_{\delta_1}} \|v\| _{H^{s_2}_{\delta_2}}.$$
\end{prp}

The following simple lemma will be useful as well.

\begin{lm}\label{chap2:produit2}
Let $\alpha \in \mathbb{R}$ and $g \in L^\infty_{loc}$ be such that
$$|g(x)| \lesssim (1+|x|^2)^\alpha.$$
Then the multiplication by $g$ maps $H^0_{\delta}$ to $H^0_{\delta -2\alpha}$.
\end{lm}

We will also need the following modified version of Lemma \eqref{chap2:produit2}.

\begin{lm}\label{chap2:produit3}
Let $\alpha \in \mathbb{R}$ and $g_1 \in L^\infty_{loc}$ be a  function such that
$$|g_1(x)| \lesssim (1+|x|^2)^\alpha.$$
Let $g_2\in L^2(\m S^1)$. Then the multiplication by $g_1(x)g_2(\theta)$ maps $H^1_{\delta}$ to $H^0_{\delta -2\alpha}$.
\end{lm}
\begin{proof}
Let $u\in H^1_{\delta}$. We estimate
\begin{align*}
&\int_0^\infty \int_0^{2\pi} (1+r^2)^{\delta-2\alpha}g_1(x)^2g_2(\theta)^2u(r,\theta)^2rdrd\theta\\
&\leq \|g_2\|^2_{L^2(\m S^1)}\int_0^\infty (1+r^2)^{\delta-2\alpha}\left(\sup_{\theta \in [0,2\pi]}|g_1|(r,\theta)\right)^2\left(\sup_{\theta \in [0,2\pi]}|u|(r,\theta)\right)^2rdr\\
&\lesssim \|g_2\|^2_{L^2(\m S^1)}\int_0^\infty(1+r^2)^{\delta}\left(\int_0^{2\pi} |u|^2 +|\partial_\theta u|^2d\theta\right) rdr\\
&\lesssim \|g_2\|^2_{L^2(\m S^1)}\left(\int (1+r^2)^{\delta}u^2 dx + \int (1+r^2)^{\delta+1}|\nabla u|^2dx\right)\\
&\lesssim \|g_2\|^2_{L^2(\m S^1)}\|u\|^2_{H^1_{\delta}}
\end{align*}
where we have used the Sobolev embedding of $L^\infty(\m S^1)$ in the Sobolev space $W^{1,2}(\m S^1)$.
\end{proof}

We will use the following definition
\begin{df}
Let $\delta\in \m R$ and $s \in \m N$. We note $\q H^s_{\delta}$ the set of symmetric traceless 2-tensors whose components are in 
$H^s_{\delta}$.
\end{df}

\subsection{Behavior of the Laplace operator in weighted Sobolev spaces.}
\begin{thm}\label{chap2:laplacien}(Theorem 0 in \cite{laplacien})
 Let $m\in \mathbb{N}$ and $-1+m<\delta<m$. The Laplace operator $\Delta :H^2_\delta \rightarrow H^0_{\delta+2}$ is an injection with closed range 
$$\left \{f \in H^0_{\delta+2}\; | \;\int fv =0 \quad \forall v \in \cup_{i=0}^m \mathcal{H}_i \right \},$$
where $\mathcal{H}_i$ is the set of harmonic polynomials of degree $i$.
Moreover, $u$ obeys the estimate
$$\|u\|_{H^2_{\delta}} \leq C(\delta)\|\Delta u\|_{H^0_{\delta+2}},$$
where $C(\delta)$ is a constant such that $C(\delta) \rightarrow +\infty$ when $\delta \rightarrow m_{-}$ and  $\delta \rightarrow (-1+m)_{+}$.
\end{thm}
The following corollary has been proved in \cite{moi}.

\begin{cor}\label{chap2:regu}
 Let $s,m \in \mathbb{N}$ and $-1+m < \delta < m$. The Laplace operator $\Delta :H^{2+s}_\delta \rightarrow H^s_{\delta+2}$ is an injection with closed range 
$$\left \{f \in H^s_{\delta+2}\; | \; \int fv =0 \quad \forall v \in \cup_{i=0}^m \mathcal{H}_i \right \}.$$
Moreover, $u$ obeys the estimate
$$\|u\|_{H^{s+2}_{\delta}} \leq C(s,\delta)\|\Delta u\|_{H^s_{\delta+2}}.$$
\end{cor}

We now prove the following two corollaries of Theorem \ref{chap2:laplacien}
which will be fundamental in our work.

\begin{cor}\label{chap2:coro1}
Let $-1<\delta<0$.
Let $f\in H^0_{\delta+3}$. Then there exists a solution $u$ of 
$$\Delta u = f,$$
which can be written uniquely in the form
$$u= \frac{1}{2\pi}\left(\int f \right)\chi(r)\ln(r)-\frac{1}{2\pi}\left(\cos(\theta)\int fx_1 +\sin(\theta)\int fx_2\right)\frac{\chi(r)}{r}+ \wht u,$$
where $\wht u \in H^2_{\delta+1}$. Moreover, we have the estimate
$$\|\tilde{u}\|_{H^2_{\delta+1}}\lesssim C(\delta)\|f\|_{H^0_{\delta+3}}.$$

\end{cor}

\begin{proof}
Let $F$ be a radial function, smooth, compactly supported, such that $\int F=2\pi$, and $G$ a radial function, smooth, compactly supported, which is $0$ in a neighborhood of $0$ and such that $\int Gr = 4\pi$. We note 
$$G_1(x)=G(r)\cos(\theta)\;  and \; G_2(x)=G(r) \sin(\theta).$$
Let 
$$u_0(x)=\frac{1}{2\pi}\int F(y)\ln(|x-y|)dy$$
be a solution of $\Delta u_0=F$, and 
$$u_i(x)=\frac{1}{2\pi}\int G_i(y)\ln(|x-y|)dy$$
be a solution of $\Delta u_i = G_i$. We may calculate
\begin{align*}
u_0&= \chi(r)\ln(r) + \wht u_0,\\
u_1&=-\chi(r)\frac{\cos(\theta)}{r} +\wht u_1,\\
u_2&=-\chi(r)\frac{\sin(\theta)}{r} +\wht u_2,
\end{align*}
where $\wht u_0,\wht u_i \in H^2_{\delta+1}$.

Thanks to Theorem \eqref{chap2:laplacien}, we can solve the equation
$$\Delta v= f-\frac{1}{2\pi}\left(\int f\right)F
-\frac{1}{2\pi}\left(\int f x_1 \right)G_1-\frac{1}{2\pi}\left(\int f x_2\right) G_2$$
since the right-hand side is orthogonal to the polynomials of degree $0$ and $1$, and we have $v\in H^2_{\delta+1}$, which satisfies
$$\|v\|_{H^2_{\delta+1}}\lesssim \|f\|_{H^0_{\delta+3}}+\int |f|+\int r|f|
\lesssim \|f\|_{H^0_{\delta+3}}+\int |f|\frac{(1+r^2)^{\frac{\delta}{2}+\frac{3}{2}}}{(1+r^2)^{\frac{\delta}{2}+1}}
\lesssim \frac{1}{\sqrt{\delta+1}}\|f\|_{H^0_{\delta+3}}
.$$
Therefore we can solve the equation $\Delta u=f$, and $u$ can be written
\begin{align*}
u&=v+\frac{1}{2\pi}\left(\int f\right)+\frac{1}{2\pi}\left(\int fx_1\right)u_1+\frac{1}{2\pi}\left(\int fx_2\right)u_2\\
&=\frac{1}{2\pi}\left(\int f \right)\chi(r)\ln(r)-\frac{1}{2\pi}\left(\cos(\theta)\int fx_1 +\sin(\theta)\int fx_2\right)\frac{\chi(r)}{r}+ \wht u,
\end{align*}
where $\wht u \in H^2_{\delta+1}$ with
$$\|\tilde{u}\|_{H^2_{\delta+1}}\lesssim \|f\|_{H^0_{\delta+3}}.$$
This concludes the proof of Corollary \ref{chap2:coro1}.
\end{proof}

We introduce the notation.
$$M_\theta = \left(\begin{array}{ll}\cos(2\theta)&\sin(2\theta)\\\sin(2\theta)&-\cos(2\theta)\end{array}\right), \quad N_\theta=\left(\begin{array}{ll}-\sin(2\theta)&\cos(2\theta)\\\cos(2\theta)&\sin(2\theta)\end{array}\right).$$

\begin{cor}\label{chap2:coro2}
Let $-1<\delta<0$. Let $f_j \in H^0_{\delta + 3}$ with $\int f_j = 0$, $j=1,2$. Then, there exists a symmetric and traceless 2-tensor $K$ solution of
$$\partial_i K_{ij}=f_j,$$
which can be written uniquely in the form
$$K=A\frac{\chi(r)}{r^2}M_\theta
+J\frac{\chi(r)}{r^2}N_\theta +\wht K,$$
with $\wht K \in \q H^1_{\delta + 2}$ and
$$
A=\frac{1}{2\pi}\int x_1f_1+x_2f_2, \quad J=\frac{1}{2\pi}\int x_1f_2-x_2f_1,\quad \|\wht K\|_{\q H^1_{\delta+2}}\lesssim \|f_1\|_{H^0_{\delta+3}}+\|f_1\|_{H^0_{\delta+3}}.
$$
\end{cor}
\begin{proof}
We can look for $K$ of the form
$$K_{ij}=\partial_i Y_j+\partial_j Y_i -\delta_{ij}\partial^k Y_k,$$
then $Y_j$ satisfies
$$\Delta Y_j = f_j.$$
We apply Corollary \eqref{chap2:coro1} which allows us to find a solution in the form\footnote{Recall that $\int f_j=0$.}
$$Y_j=\frac{\chi(r)}{r}\left(a_j\cos(\theta)+b_j \sin(\theta)\right)+\wht Y_j,$$
with $\wht Y_j \in H^2_{\delta+1}$ and
$$a_j = -\frac{1}{2\pi}\int x_1 f_j, \quad b_j=-\frac{1}{2\pi}\int x_2 f_j, \quad \|\wht Y_j\|_{H^2_{\delta+1}}\lesssim \|f_j\|_{H^0_{\delta+3}}.$$
We calculate
\begin{align*}
K_{11}=&\partial_1 Y_1-\partial_2 Y_2\\
=&\frac{\chi(r)}{r^2}\Big(a_1(-\cos^2(\theta)+\sin^2(\theta))
-2b_1\cos(\theta)\sin(\theta)+2a_2\cos(\theta)\sin(\theta)\\
&+b_2(\sin^2(\theta)-\cos^2(\theta))\Big)+\wht K_{11}\\
=&\frac{\chi(r)}{r^2}\left(-(a_1+b_2)\cos(2\theta)+(a_2-b_1)\sin(2\theta)\right) +\wht K_{11},\\
K_{12}=&\partial_1Y_2+\partial_2Y_1\\
=&\frac{\chi(r)}{r^2}\Big(a_2(-\cos^2(\theta)+\sin^2(\theta))-2b_2\cos(\theta)\sin(\theta)-2a_1\cos(\theta)\sin(\theta)\\
&+b_1(-\sin^2(\theta)+\cos^2(\theta))\Big)+\wht K_{12}\\
=&\frac{\chi(r)}{r^2}\left(-(a_2-b_1)\cos(2\theta)-(a_1+b_2)\sin(2\theta)\right)+\wht K_{12},
\end{align*}
with $\wht K_{11}$ and $\wht K_{12}$ in $H^1_{\delta+2}$ and
$$\| \wht K_{11}\|_{H^1_{\delta+2}}+\| \wht K_{12}\|_{H^1_{\delta+2}}\lesssim \|f_1\|_{H^0_{\delta+3}}+\|f_2\|_{H^0_{\delta+3}}.$$
Therefore we can write
$$K=A\frac{\chi(r)}{r^2}\left(\begin{array}{ll}\cos(2\theta)&\sin(2\theta)\\\sin(2\theta)&-\cos(2\theta)\end{array}\right)
+J\frac{\chi(r)}{r^2}\left(\begin{array}{ll}-\sin(2\theta)&\cos(2\theta)\\\cos(2\theta)&\sin(2\theta)\end{array}\right) +\wht K,$$
with
\begin{align*}
A&=-(a_1+b_2)=\frac{1}{2\pi}\int x_1f_1+x_2f_2,\\
J&=b_1-a_2=\frac{1}{2\pi}\int x_1f_2-x_2f_1,\\
\|\wht K\|_{\q H^1_{\delta+2}}&\lesssim \|f_1\|_{H^0_{\delta+2}}+\|f_2\|_{H^0_{\delta+2}}.
\end{align*}
This concludes the proof of Corollary \ref{chap2:coro2}.
\end{proof}

\section{Main result and outline of the proof}
In \cite{moi}, we solved the system \eqref{chap2:sys} for $\dot{u}^2, |\nabla u|^2 \in H^0_{\delta+2}$ with $-1<\delta<0$. The solutions we found were of the form
\begin{align*}
\lambda&= -\alpha\chi(r)\ln(r)+\wht{\lambda},\\
H&= -(\rho\cos(\theta-\eta)+\wht{b}(\theta))\frac{\chi(r)}{2r}M_\theta +\wht{H},\\
\tau&=(\rho\cos(\theta-\eta)+\wht{b}(\theta))\frac{\chi(r)}{r}+\wht{\tau},
\end{align*}
where $\wht{\lambda}\in H^2_{\delta}$, $\wht H \in \q H^1_{\delta+1}$. By looking for $H$ as $H_{ij}=\partial_i Y_j+\partial_j Y_i -\delta_{ij}\partial^kY_k$, 
the system  \eqref{chap2:sys} corresponds to three Laplace-like equations.
The quantities $\wht b \in W^{1,2}(\m S^1)$ and $\wht \tau \in H^1_{\delta+1}$ are free parameters, while
the three parameters $\alpha$, $\rho$ and $\eta$ are determined by the three corresponding orthogonality conditions, namely that the integrals of the right-hand sides of \eqref{chap2:sys} vanish.

In this paper, assuming that $\dot{u}^2, |\nabla u|^2 \in H^0_{\delta+3}$ (i.e. assuming more decay on $u$ and $\dot{u}$ than in \cite{moi}), we want to go further in the asymptotic expansion of our solution.
This will require to enforce additional orthogonality conditions.

\subsection{Main result}

\begin{thm}\label{chap2:main}
Let $-1<\delta<0$. Let $\dot{u}^2, |\nabla u|^2 \in H^0_{\delta+3}$ and $\wht b\in W^{1,2}(\m S^1)$ such that
\begin{equation}
\label{ortob}\int_{\m S^1} \wht b(\theta)\cos(\theta)d\theta=\int_{\m S^1} \wht b(\theta)\sin(\theta)d\theta =0.
\end{equation}
We note
$$\varepsilon = \int \dot{u}^2 +|\nabla u|^2.$$
We assume 
$$\|\dot{u}^2\| _{H^0_{\delta+3}} + \||\nabla u|^2\|_{H^0_{\delta+3}} +\|\wht b\|_{W^{1,2}}\lesssim \varepsilon.$$
Let $B\in W^{1,2}(\m S^1)$. We assume
$$\| B\|_{W^{1,2}}\lesssim \varepsilon^2.$$
Let $\Psi \in H^1_{\delta+2}$ be such that $\int \Psi = 2\pi$.
If $\ep>0$ is small enough, there exist $\alpha,\rho,\eta,A,J,c_1,c_2$ in $\m R$, a scalar functions $\wht \lambda \in H^2_{\delta+1}$
and a symmetric traceless tensor $\wht H \in \q H^1_{\delta+2}$ such that, if
$r,\theta$ are the polar coordinates centered in $(c_1,c_2)$, and if we note
\begin{align*}
\lambda&= -\alpha \chi(r)\ln(r)+\wht \lambda,\\
H&=-(\wht b(\theta)+\rho\cos(\theta-\eta))\frac{\chi(r)}{2r}M_\theta + e^{-\lambda}\frac{\chi(r)}{r^2}\left((J-(1-\alpha)B(\theta))N_\theta -\frac{B'(\theta)}{2}M_\theta \right) + \wht H,
\end{align*}
then $\lambda, H$ are solutions of \eqref{chap2:sys} with
$$\tau=(\wht b(\theta)+\rho \cos(\theta-\eta))\frac{\chi(r)}{r}+e^{-\lambda} B'(\theta)\frac{\chi(r)}{r^2}+A\Psi.$$
Moreover we have the estimates
\begin{align*}
\alpha&=\frac{1}{4\pi}\int \left(\dot{u}^2+|\nabla u|^2\right) +O(\ep^2),\\
\rho \cos(\eta)&=\frac{1}{\pi}\int \dot{u}.\partial_1 u +O(\ep^2),\\
\rho \sin(\eta)&=\frac{1}{\pi}\int \dot{u}.\partial_2 u +O(\ep^2),\\
c_1&=\frac{1}{4\pi}\int x_1\left(\dot{u}^2+|\nabla u|^2\right) +O(\ep^2),\\
c_2&=\frac{1}{4\pi}\int x_2\left(\dot{u}^2+|\nabla u|^2\right) +O(\ep^2),\\
J&=-\frac{1}{2\pi}\int \dot{u}.\partial_\theta u+\frac{\rho }{2\alpha}(c_1\sin(\eta)-c_2\cos(\eta))+O(\ep^2),\\
A&=-\frac{1}{2\pi}\int r\dot{u}.\partial_r u
+\frac{1}{2\pi}\left(\int \chi'(r)rdr\right)\int \wht b(\theta)d\theta + O(\ep^2),
\end{align*}
and
$$\|\wht \lambda\|_{H^2_{\delta+1}}+\|\wht \tau\|_{H^1_{\delta+2}}+\|\wht H\|_{\q H^1_{\delta+2}}\lesssim \ep.$$
\end{thm}

\begin{rk}
There is a natural rapprochement between the quantities
$\alpha, \rho, \eta, c_1,c_2, J,A$ and the global charges in $3+1$ dimensions (such as the ADM mass, ADM momentum...). See for example \cite{cor} for a definition.
\end{rk}

The following corollary is a straightforward consequence of Theorem \eqref{chap2:main} and Corollary \eqref{chap2:regu}.

\begin{cor}\label{chap2:regu2}Let $\delta,\dot{u},\nabla u, \ep, \wht b, B $ and $\Psi$ be as in the assumptions of Theorem \ref{chap2:main}. Moreover
let $s\in \mathbb{N}$ and assume $\dot{u}^2, |\nabla u|^2 \in H^s_{\delta+3}$, $B,\wht b \in W^{s+1,2}(\m S^1)$
and $\Psi \in H^{s+1}_{\delta+2}$. 
Then the conclusion of Theorem \eqref{chap2:main} holds and we have furthermore $\widetilde{\lambda} \in H^{s+2}_{\delta+1}$,
 $\widetilde{H} \in \q H^{s+1}_{\delta+2}$, with the estimates 
$$\|\widetilde{\lambda}\|_{H^{s+2}_{\delta+1}}+\|\widetilde{H}\|_{\q H^{s+1}_{\delta+2}} \lesssim 
\|\dot{u}^2\|_{H^s_{\delta+3}}+\||\nabla u|^2\|_{H^s_{\delta+3}} +\|\wht b\|_{W^{s+1,2}}+\|B\|_{W^{s+1,2}}.$$
\end{cor}

\subsection{Outline of the proof}
We will prove the theorem using a fixed point argument.
\paragraph{Construction of the map $F$}
We consider the map
\begin{align*}
F:\m R\times \m R \times \m R \times H^2_{\delta+1}&\rightarrow \m R \times \m R \times \m R \times H^2_{\delta+1}\\
(\alpha, c_1, c_2, \wht \lambda)&\mapsto (\alpha', c_1', c_2', \wht \lambda')
\end{align*}
where if we note
$$(c_1,c_2)=r_c(\cos(\theta_c),\sin(\theta_c)), \quad (c'_1,c'_2)=r'_c(\cos(\theta'_c),\sin(\theta'_c))$$
and
\begin{align*}
\lambda= &-\alpha\chi(r)\ln(r)+ r_c\cos(\theta-\theta_c)\frac{\chi(r)}{r}+\wht \lambda\\
\lambda'=& -\alpha'\chi(r)\ln(r)+ r'_c\cos(\theta-\theta'_c)\frac{\chi(r)}{r}+\wht \lambda',
\end{align*}
then $\lambda'$ is the solution of
\begin{equation}
\label{chap2:eqlamb}
\Delta \lambda' + \frac{1}{2}\dot{u}^2 + \frac{1}{2}|\nabla u|^2+\frac{1}{2}|H|^2-\frac{\tau^2}{4} = 0,
\end{equation}
with
\begin{equation}
\label{chap2:defH}H=e^{-\lambda}H^{(1)}+H^{(2)}+e^{-\lambda}H^{(3)},
\end{equation}
where
\begin{align}
\label{chap2:defh2}H^{(2)}&=-b(\theta)\frac{\chi(r)}{2r}M_\theta 
-\frac{r_c}{2\alpha}(b(\theta)\sin(\theta-\theta_c))'
\frac{\chi(r)}{r^2}M_\theta
-\frac{r_c}{\alpha}b(\theta)\sin(\theta-\theta_c)\frac{\chi(r)}{r^2}N_\theta, \\
\label{chap2:defh3}H^{(3)}&=\frac{\chi(r)}{r^2}\left(-(1-\alpha)B(\theta)N_\theta -\frac{B'(\theta)}{2}M_\theta \right) ,
\end{align}
and $H$ satisfies
\begin{equation}\label{chap2:eqh}
\partial_i H_{ij} +H_{ij}\partial_i \lambda  = -\dot{u}.\partial_j u + \frac{1}{2} \partial_j \tau-\frac{1}{2} \tau\partial_j \lambda,
\end{equation}
with
\begin{equation}
\label{chap2:deftau}\tau=\tau^{(2)}+e^{-\lambda}\tau^{(3)}+A\Psi,
\end{equation}
where
\begin{align}
\label{chap2:deftau2}\tau^{(2)}&= b(\theta)\frac{\chi(r)}{r}+\frac{r_c}{\alpha}(b(\theta)\sin(\theta-\theta_c))'\frac{\chi(r)}{r^2},\\
\label{chap2:deftau3}\tau^{(3)}&=B'(\theta)\frac{\chi(r)}{r^2}.
\end{align}
We have noted 
\begin{equation}
\label{chap2:defb}
b(\theta)=\rho\cos(\theta-\eta)+ \wht b(\theta).
\end{equation}
The parameters $\rho,\eta$ and $A$ are suitably chosen during the process.
\paragraph{Solving \eqref{chap2:eqh}}
We will show that $H^{(1)}$ satisfies
$$\partial_i H^{(1)}_{ij}=f^{(1)}_j$$
with $f^{(1)}_j \in H^0_{\delta+3}$. We will prove that we may choose $\rho,\eta$ and $A$ such that
$$\int f^{(1)}_1=\int f^{(2)}_2=\int x_1f^{(1)}_1+x_2f^{(1)}_2=0.$$
Then we will show that $H^{(1)}$ can be written
$$H^{(1)}=J\frac{\chi(r)}{r^2}N_\theta+\wht H^{(1)},$$
with $\wht H^{(1)} \in H^1_{\delta+2}$.

\paragraph{Solving \eqref{chap2:eqlamb}}
We will show that 
$$\frac{1}{2}|H|^2-\frac{\tau^2}{4}\in H^0_{\delta+3}.$$
Then, it will be straightforward to solve \eqref{chap2:eqlamb} using Corollary \eqref{chap2:coro1}.
The solution we obtain is of the form 
$$ \lambda'=-\alpha'\chi(r)\ln(r)+ r'_c\cos(\theta-\theta'_c)\frac{\chi(r)}{r}+\wht \lambda',$$
with $\wht \lambda'\in H^2_{\delta+1}$.
\paragraph{The fixed point}
Proving that $F$ is a contracting map easily follows from the estimates for $\lambda'$ and $H$. The Picard fixed point theorem then implies that $F$ has a fixed point.
To obtain the result stated in Theorem \eqref{chap2:main} then easily folows after performing the following change of variables
$$x_1'=x_1-\frac{r_c}{\alpha}\cos(\theta_c), \quad x_2'=x_2+\frac{r_c}{\alpha}\sin(\theta_c),$$
which corresponds to work in a frame centered in the center of mass.
\\

The rest of the paper is as follows. In section \eqref{chap2:secmom}, we explain how to solve the momentum constraint \eqref{chap2:eqh}. We also explain how to choose $A,\rho,\eta$. In section \eqref{chap2:seclic}, we explain how to solve \eqref{chap2:eqlamb}. Finally, the map $F$ is shown to have a fixed point in section \eqref{chap2:secpf}.

\section{The momentum constraint}\label{chap2:secmom}
 The goal of this section is to solve equation \eqref{chap2:eqh}.
We will note
\begin{equation}
\label{chap2:normel}\|\lambda\|=|\alpha|+r_c+\|\wht \lambda\|_{H^2_{\delta+1}}.
\end{equation}
We assume a priori
$$\|\lambda\|\lesssim \ep,$$
\begin{equation}
\label{chap2:aprio}
\alpha \geq \frac{1}{8\pi}\left(\int \dot{u}^2+|\nabla u|^2\right).
\end{equation}
This yields
$$\frac{r_c}{\alpha}\lesssim \frac{\|\lambda\|}{\ep}\lesssim 1.$$

 \begin{prp}\label{chap2:prpeqh}If $\ep>0$ is small enough, there exists $\rho,\eta, A \in \m R$, such that 
 for
$$\tau=\tau^{(2)}+e^{-\lambda}\tau^{(3)}+A\Psi,$$
with $\tau^{(2)}, \tau^{(3)}$ defined by \eqref{chap2:deftau2} and \eqref{chap2:deftau3},
 there
 exists a solution of \eqref{chap2:eqh} which may be uniquely written under the form
 $$H=e^{-\lambda}H^{(1)}+H^{(2)}+e^{-\lambda}H^{(3)}$$
 where $H^{(2)}$ and $H^{(3)}$ are defined by \eqref{chap2:defh2} and \eqref{chap2:defh3} and 
$$H^{(1)}=J\frac{\chi(r)}{r^2}N_\theta +\wht H^{(1)},$$
with $e^{-\lambda}\wht H^{(1)} \in \q H^1_{\delta+2}$ such that
$$\|e^{-\lambda} \wht H^{(1)} \|_{\q H^1_{\delta+2}} \lesssim\|\dot{u}\nabla u \|_{H^0_{\delta+3}}+\|b\|_{W^{1,2}}+ \|B\|_{W^{1,2}}+|A|\lesssim \ep.$$
Moreover we have the estimates
\begin{align*}
\rho \cos(\eta)&=\frac{1}{\pi}\int \dot{u}.\partial_1 u +O(\ep^2),\\
\rho \sin(\eta)&=\frac{1}{\pi}\int \dot{u}.\partial_2 u +O(\ep^2),\\
A&=-\frac{1}{2\pi}\int r\dot{u}.\partial_r u
+\frac{1}{2\pi}\left(\int \chi'(r)rdr\right)\int \wht b(\theta)d\theta + O(\ep^2),\\
J&=-\frac{1}{2\pi}\int \dot{u}.\partial_\theta u+\frac{\rho r_c}{2\alpha}\sin(\eta-\theta_c)+O(\ep^2).
\end{align*}
 \end{prp}
 \begin{proof}
We introduce the notation
\begin{align}
\label{chap2:defhj2}
h_j^{(2)}&=-\frac{1}{2}\tau^{(2)}\partial_j \lambda -H^{(2)}_{ij}\partial_j \lambda,  \\
\label{chap2:defgj2}
h_j^{(3)}&=\frac{1}{2}\partial_j \left(e^{-\lambda}\tau^{(3)}\right) 
-\frac{1}{2}e^{-\lambda}\tau^{(3)}\partial_j \lambda-
\partial_i(e^{-\lambda} H^{(3)})_{ij}-e^{-\lambda}H^{(3)}_{ij}\partial_j\lambda.
\end{align}
In view of \eqref{chap2:defH}, \eqref{chap2:eqh} and \eqref{chap2:deftau} an easy calculation yields
 \begin{equation}
 \label{chap2:eqfj}\partial_i H_{ij}^{(1)}= f^{(1)}_j
 \end{equation}
where
\begin{equation}
\label{deffj} f^{(1)}_j= e^{\lambda}\Big(-\dot{u}.\partial_j u+\frac{1}{2}\partial_j(A\Psi)-\frac{1}{2}A\Psi\partial_j \lambda
 + h_j^{(2)}+h_j^{(3)} 
+\frac{1}{2}\partial_j \tau^{(2)} -\partial_i H^{(2)}_{ij}\Big).
\end{equation}
 The three following propositions, proved respectively in Sections \eqref{chap2:sech1}, \eqref{chap2:sech2} and \eqref{chap2:sech3} allow us to estimate the different contributions to $f_j^{(1)}$.
\begin{prp}\label{chap2:prph1}
We have
\begin{align*}
\frac{1}{2}\partial_1 \tau^{(2)} -\partial_i H^{(2)}_{i1}
=&\frac{\chi'(r)}{r}b(\theta)\cos(\theta)+\frac{\chi'(r)}{r^2}\frac{r_c}{\alpha}\left(b(\theta)\sin(\theta-\theta_c)\cos(\theta)\right)',\\
\frac{1}{2}\partial_2\tau^{(2)} -\partial_i H^{(2)}_{i2}
=&\frac{\chi'(r)}{r}b(\theta)\sin(\theta)
+\frac{\chi'(r)}{r^2}\frac{r_c}{\alpha}\left(b(\theta)\sin(\theta-\theta_c)\sin(\theta)\right)'.
\end{align*}
\end{prp}

\begin{prp}\label{chap2:prph2}
We have $h^{(2)}_j \in H^0_{\delta+3}$, with
$$\|h^{(2)}_j \|_{H^0_{\delta+3}} \lesssim  \|\lambda\|\|b\|_{W^{1,2}}.$$
\end{prp}
\begin{prp}\label{chap2:prph3}
We have
$h^{(3)}_j\in H^0_{\delta+3},$ with
$$\|h^{(3)}_j \|_{H^0_{\delta+3}}\lesssim \|B\|_{W^{1,2}}.$$
\end{prp}
We have $e^{-\lambda}f_j^{(1)} \in H^0_{\delta + 3}$ :
\begin{itemize}
\item For $h_j^{(2)}$ and $h_j^{(3)}$ this follows from Propositions \eqref{chap2:prph2} and \eqref{chap2:prph3}.
\item For $\frac{1}{2}\partial_j\tau^{(2)} -\partial_i H^{(2)}_{j2}$, this is a consequence of Proposition $\eqref{chap2:prph1}$. Since $\chi'$ is compactly supported, we have
$$\left\|\frac{1}{2}\partial_j\tau^{(2)} -\partial_i H^{(2)}_{j2}\right\|_{H^0_{\delta+3}}\lesssim \|b\|_{W^{1,2}(\m S^1)}.$$
\item Since $\Psi \in H^1_{\delta +2}$, we have in view of Lemma \ref{chap2:der}
$$\|A\partial_j \Psi\|_{H^0_{\delta+3}} \lesssim |A|.$$
\item We have
\begin{align*}A\Psi\partial_j \lambda=&\left(-\alpha\frac{\chi(r)}{r}- r_c\cos(\theta-\theta_c)\frac{\chi(r)}{r^2}
-\alpha \chi'(r)\ln(r)+r_c\cos(\theta-\theta_c)\frac{\chi'(r)}{r}\right)A\Psi \partial_j r\\
&- r_c\sin(\theta-\theta_c)\frac{\chi(r)}{r}A\Psi \partial_j \theta +A\Psi\partial_j \wht \lambda,
\end{align*}
and since $\chi'$ is compactly supported we have
$$\|A\Psi\partial_j \lambda\|_{H^0_{\delta+3}} \lesssim \left\|A\Psi \frac{\alpha}{1+r}\right\|_{H^0_{\delta+3}}+
\left\|A\Psi \frac{r_c}{1+r^2}\right\|_{H^0_{\delta +3}}+\| A\Psi \partial_j \wht \lambda\|_{H^0_{\delta+3}}.
+|A|.$$
For the terms of the form $A\Psi \frac{\alpha}{1+r}$ and $A\Psi \frac{r_c}{1+r^2}$, we use Lemma \eqref{chap2:produit2} which yields
$$\left\|A\Psi \frac{\alpha}{1+r}\right\|_{H^0_{\delta+3}}\lesssim |A||\alpha|, \quad
\left\|A\Psi \frac{r_c}{1+r^2}\right\|_{H^0_{\delta +3}}\lesssim |A||r_c|.$$
\item For the term $A\Psi \partial_j \wht \lambda$ we use Proposition \eqref{chap2:produit} which yields
$$\| A\Psi \partial_j \wht \lambda\|_{H^0_{\delta+3}}\lesssim |A|\|\wht \lambda\|_{H^2_{\delta+1}}.$$
\end{itemize}
Consequently we have
$$\|e^{-\lambda}f^{(1)}_j\|_{H^0_{\delta+3}}\lesssim \|\dot{u}\nabla u \|_{H^0_{\delta+3}}+\|b\|_{W^{1,2}}+ \|B\|_{W^{1,2}}+|A|.$$
We have 
$$\lambda= -\alpha\chi(r)\ln(r)+\frac{r_c\cos(\theta-\theta_c)\chi(r)}{r}+\wht \lambda,$$
with $\wht \lambda \in H^2_{\delta+1}\subset L^\infty$ thanks to Proposition \eqref{chap2:holder}. Therefore
$$|e^{\lambda}|\lesssim (1+r^2)^{-\frac{\alpha}{2}},$$
and Lemma \eqref{chap2:produit2} yields $f^{(1)}_j \in H^0_{\delta+3+\alpha}$ with
$$\|f^{(1)}_j\|_{H^0_{\delta+3+\alpha}}\lesssim  \|\dot{u}\nabla u \|_{H^0_{\delta+3}}+\|b\|_{W^{1,2}}+ \|B\|_{W^{1,2}}+|A|.
$$
We want to solve \eqref{chap2:eqfj} with Corollary \eqref{chap2:coro2}. To this end, we need
\begin{equation}
\label{condint}\int f^{(1)}_1 = \int f^{(1)}_2=0.
\end{equation}
The following proposition, proven in Section \eqref{chap2:sech4}, allows us to carefully choose the parameters $\rho,\eta, A$
in order to enforce the orthogonality condition \eqref{condint}.
\begin{prp}\label{chap2:prph4}
If $\ep>0$ is small enough, there exist $\rho, \eta, A \in \m R$ such that
$$\int f^{(1)}_1 = \int f^{(1)}_2=\int x_1 f^{(1)}_1+x_2f^{(2)}_2=0.$$
Moreover we have
\begin{align*}
\rho \cos(\eta)&=\frac{1}{\pi}\int e^{\lambda}\dot{u}\partial_1 u +O(\ep^2),\\
\rho \sin(\eta)&=\frac{1}{\pi}\int e^{\lambda}\dot{u}\partial_2 u +O(\ep^2),\\
A&=-\frac{1}{2\pi}\int e^{\lambda}\dot{u}r\partial_r u
+\frac{1}{2\pi}\left(\int \chi'(r)rdr\right)\int b(\theta)d\theta + O(\ep^2).
\end{align*}
\end{prp}
We choose $\rho, \eta, A$ according to Proposition \eqref{chap2:prph4}. Since $|\alpha|\lesssim \ep$, if $\ep>0$
is small enough, we have $-1<\delta+\alpha<0$.
Since $\int f^{(1)}_1=\int f^{(1)}_2=0$, we can apply Corollary \eqref{chap2:coro2}.
Since $\int x_1f^{(1)}_1+x_2f^{(1)}_2=0$, we obtain
$$H^{(1)}=J\frac{\chi(r)}{r^2}N_\theta +\wht H^{(1)},$$
with $\wht H^{(1)} \in \q H^1_{\delta+2+\alpha}$ such that
$$\left\|\wht H^{(1)} \right\|_{H^1_{\delta+2+\alpha}}\lesssim \left\|f_1^{(1)}\right\|_{H^0_{\delta+3+\alpha}} + \left\|f_2^{(1)}\right\|_{H^0_{\delta+3+\alpha}} \lesssim \|\dot{u}\nabla u \|_{H^0_{\delta+3}}+\|b\|_{W^{1,2}}+ \|B\|_{W^{1,2}}+|A|,$$
and
\begin{equation}\label{chap2:defJ}\begin{split}
J=&\frac{1}{2\pi}\int x_1f^{(1)}_2-x_2f^{(1)}_1\\
=&\frac{1}{2\pi}
\int_{\m R^2} e^{\lambda}\left(-\dot{u}\partial_\theta u
-A\Psi \partial_\theta \lambda +x_1(h^{(2)}_2+h^{(3)}_2)-x_2(h^{(2)}_1+h^{(3)}_1)\right)
+\frac{\rho r_c}{2\alpha}\sin(\eta-\theta_c)\\
&+\frac{r_c}{\alpha}\int (e^{\lambda}-1)\frac{\chi'(r)}{r}b(\theta)\sin(\theta-\theta_c)\\
=&-\frac{1}{2\pi}\int_{\m R^2} e^{\lambda}\dot{u}\partial_\theta u+\rho \frac{r_c}{2\alpha}\sin(\theta-\theta_c)+O(\ep^2)
\end{split}
\end{equation}
where we have used the definition \eqref{deffj} of $f^{(1)}_j$, $x_1\partial_2-x_2\partial_1= \partial_\theta$, Proposition \ref{chap2:prph1} and the following calculations
$$\frac{1}{2}\int e^{\lambda} A\partial_\theta \Psi=-\frac{1}{2}\int e^{\lambda}A\Psi\partial_\theta \lambda,$$
\begin{align*}
&\int e^\lambda\left(x_1\frac{\chi'(r)}{r}b(\theta)\sin(\theta)-x_2\frac{\chi'(r)}{r}b(\theta)\cos(\theta)\right)\\=&\int e^{\lambda}\chi'(r) b(\theta)(\cos(\theta)\sin(\theta)-\cos(\theta)\sin(\theta))rdrd\theta\\
=& 0,
\end{align*}
\begin{align*}
&\int e^{\lambda}\frac{\chi'(r)}{r^2}\frac{r_c}{\alpha}\Bigg( x_1\left(b(\theta)\sin(\theta-\theta_c)\sin(\theta)\right)'
-x_2\left(b(\theta)\sin(\theta-\theta_c)\cos(\theta)\right)'\Bigg)\\
=&-\frac{r_c}{\alpha}\int e^{\lambda}\frac{\chi'(r)}{r}b(\theta)\sin(\theta-\theta_c)
(-\sin^2(\theta)-\cos^2(\theta))\\
&-\frac{r_c}{\alpha}\int\partial_\theta \lambda e^\lambda \frac{\chi'(r)}{r}b(\theta)\sin(\theta-\theta_c)(\cos(\theta)\sin(\theta)-\sin(\theta)\cos(\theta))
\\
=&\frac{r_c}{\alpha}\left(\int \chi'(r)dr\right)\left(\int b(\theta)\sin(\theta-\theta_c)d\theta \right)
+\frac{r_c}{\alpha}\int (e^{\lambda}-1)\frac{\chi'(r)}{r}b(\theta)\sin(\theta-\theta_c)\\
=& \pi \frac{\rho r_c}{\alpha}\sin(\eta-\theta_c)
+\frac{r_c}{\alpha}\int (e^{\lambda}-1)\frac{\chi'(r)}{r}b(\theta)\sin(\theta-\theta_c),
\end{align*}
where we have used in the last equality the definition of $b$ \eqref{chap2:defb} and the orthogonality condition \eqref{ortob} for $\wht b$.
It remains to estimate $e^{-\lambda }\wht H^{(1)}$ in $\q H^1_{\delta+2}$. First, we note that since 
$$-\lambda-\alpha\chi(r)\ln(r)$$
 is bounded, thanks to Lemma \eqref{chap2:produit2} and the fact that $\wht H^{(1)}\in \q H^1_{\delta+2+\alpha}$ we have
$e^{-\lambda} \wht H^{(1)} \in \q H^0_{\delta+2}$. We now calculate $\nabla(e^{-\lambda} \wht H^{(1)} )$. The contributions are
\begin{itemize}
\item the term $e^{-\lambda}\nabla \wht H^{(1)}$ : since $\nabla \wht H^{(1)}\in \q H^0_{\delta+\alpha+3}$, we have  $e^{-\lambda}\nabla \wht H^{(1)} \in \q H^0_{\delta+3}$ thanks to Lemma \eqref{chap2:produit2},
\item the term $\frac{\alpha\chi(r)}{r}e^{-\lambda} \wht H^{(1)}$ : it also belongs to $\q H^0_{\delta+3}$ thanks to Lemma \eqref{chap2:produit2}.
\item The term $e^{-\lambda}\wht H^{(1)}\nabla \wht \lambda$ : thanks to Proposition \eqref{chap2:produit}, $\wht H^{(1)}\nabla \wht \lambda$ belong to $\q H^0_{\delta+3+\alpha}$, and therefore, thanks to Lemma \eqref{chap2:produit2},
we have $e^{-\lambda}\wht H^{(1)}\nabla \wht \lambda \in \q H^0_{\delta+3}$.
\end{itemize}
Consequently, we have $\nabla(e^{-\lambda} \wht H^{(1)} ) \in\q H^0_{\delta +3}$ and therefore $e^{-\lambda} \wht H^{(1)} \in \q H^1_{\delta+2}$ with
$$\|e^{-\lambda} \wht H^{(1)} \|_{\q H^1_{\delta+2}}\lesssim \|\dot{u}\nabla u \|_{H^0_{\delta+3}}+\|b\|_{W^{1,2}}+ \|B\|_{W^{1,2}}+|A|\lesssim \ep.$$
This concludes the proof of Proposition \eqref{chap2:prpeqh}.
\end{proof}
\subsection{Proof of Proposition \eqref{chap2:prph1}}\label{chap2:sech1}
We calculate
\begin{align*}
\partial_i \left( b(\theta)\frac{-\chi(r)}{2r}M_\theta\right)_{i1}
=&-\frac{b(\theta)}{2}\left(\frac{\chi'(r)}{r}-\frac{\chi(r)}{r^2}\right)(\cos(\theta)\cos(2\theta)+\sin(\theta)\sin(2\theta))\\
&-\frac{b(\theta)\chi(r)}{r^2}(\sin(\theta)\sin(2\theta)+\cos(\theta)\cos(2\theta))\\
&-\frac{b'(\theta) \chi(r)}{2r^2}
(-\sin(\theta)\cos(2\theta)+\cos(\theta)\sin(2\theta))\\
=&-\frac{b(\theta)\chi(r)}{2r^2}\cos(\theta)-\frac{b'(\theta)\chi(r)}{2r^2}\sin(\theta)-\frac{b(\theta)\chi'(r)}{2r}\cos(\theta),\\
\frac{1}{2}\partial_1\left(b(\theta)\frac{\chi(r)}{r}\right)=&\frac{1}{2}b(\theta)\left(\frac{\chi'(r)}{r}-\frac{\chi(r)}{r^2}\right)\cos(\theta)-\frac{1}{2}b'(\theta)\frac{\chi(r)}{r^2}\sin(\theta).
\end{align*}
Therefore
$$\frac{1}{2}\partial_1\left(b(\theta)\frac{\chi(r)}{r}\right)-\partial_i \left( b(\theta)\frac{-\chi(r)}{2r}M_\theta\right)_{i1}=\frac{b(\theta)\chi'(r)}{r}\cos(\theta).$$
For $j=2$ we obtain
\begin{align*}
\partial_i \left( b(\theta)\frac{-\chi(r)}{2r}M_\theta\right)_{i2}
=&-\frac{b(\theta)}{2}\left(\frac{\chi'(r)}{r}-\frac{\chi(r)}{r^2}\right)(\cos(\theta)\sin(2\theta)-\sin(\theta)\cos(2\theta))\\
&-\frac{b(\theta)\chi(r)}{r^2}(-\sin(\theta)\cos(2\theta)+\cos(\theta)\sin(2\theta))\\
&-\frac{b'(\theta) \chi(r)}{2r^2}
(-\sin(\theta)\sin(2\theta)-\cos(\theta)\cos(2\theta))\\
=&-\frac{b(\theta)\chi(r)}{2r^2}\sin(\theta)+\frac{b'(\theta)\chi(r)}{2r^2}\cos(\theta)-\frac{b(\theta)\chi'(r)}{2r}\sin(\theta),\\
\frac{1}{2}\partial_2\left(b(\theta)\frac{\chi(r)}{r}\right)=&\frac{1}{2}b(\theta)\left(\frac{\chi'(r)}{r}-\frac{\chi(r)}{r^2}\right)\sin(\theta)+\frac{1}{2}b'(\theta)\frac{\chi(r)}{r^2}\cos(\theta).
\end{align*}
Therefore
$$\frac{1}{2}\partial_2\left(b(\theta)\frac{\chi(r)}{r}\right)-\partial_i \left( b(\theta)\frac{-\chi(r)}{2r}M_\theta\right)_{i2}=\frac{b(\theta)\chi'(r)}{r}\sin(\theta).$$
We now calculate the other contributions. We note that $\frac{1}{r^2}M_\theta$ and $\frac{1}{r^2}N_\theta$ satisfy
\begin{equation}
\label{calcmn}\partial_i \left(\frac{1}{r^2}M_\theta\right)_{ij}=\partial_i \left(\frac{1}{r^2}N_\theta\right)_{ij}=0,\quad for\;r>0.
\end{equation}This yields
\begin{align*}
&\partial_i\left(-\frac{r_c}{2\alpha}(b(\theta)\sin(\theta-\theta_c))'
\frac{\chi(r)}{r^2}M_\theta
-\frac{r_c}{\alpha}b(\theta)\sin(\theta-\theta_c)\frac{\chi(r)}{r^2}N_\theta\right)_{i1}\\
=&-\frac{r_c}{2\alpha}(b(\theta)\sin(\theta-\theta_c))'(\cos(\theta)\cos(2\theta)+\sin(\theta)\sin(2\theta))\frac{\chi'(r)}{r^2}\\
&-\frac{r_c}{\alpha}b(\theta)\sin(\theta-\theta_c)(-\cos(\theta)\sin(2\theta)+\sin(\theta)\cos(2\theta))\frac{\chi'(r)}{r^2}\\
&-\frac{r_c}{2\alpha}(b(\theta)\sin(\theta-\theta_c))''(-\sin(\theta)\cos(2\theta)+\cos(\theta)\sin(2\theta))\frac{\chi(r)}{r^3}\\
&-\frac{r_c}{\alpha}(b(\theta)\sin(\theta-\theta_c))'(\sin(\theta)\sin(2\theta)+\cos(\theta)\cos(2\theta))\frac{\chi(r)}{r^3}\\
=&\frac{r_c}{\alpha}\left(b(\theta)\sin(\theta-\theta_c)\sin(\theta)-\frac{1}{2}(b(\theta)\sin(\theta-\theta_c))'\cos(\theta)\right)\frac{\chi'(r)}{r^2}\\
&+\frac{r_c}{\alpha}\left(-(b(\theta)\sin(\theta-\theta_c))'\cos(\theta)-\frac{1}{2}(b(\theta)\sin(\theta-\theta_c))''\sin(\theta)\right)
\frac{\chi(r)}{r^3}.
\end{align*}
We now calculate the term involving $\tau$.
\begin{align*}
&\frac{1}{2}\partial_1\left(\frac{r_c}{\alpha}(b(\theta)\sin(\theta-\theta_c))'\frac{\chi(r)}{r^2}\right)\\
=&\frac{r_c}{2\alpha}\left(\frac{\chi'(r)}{r^2}-2\frac{\chi(r)}{r^3}\right)\cos(\theta)(b(\theta)\sin(\theta-\theta_c))'
+\frac{r_c}{2\alpha}(-\sin(\theta))(b(\theta)\sin(\theta-\theta_c))''\frac{\chi(r)}{r^3}\\
=&\frac{r_c}{2\alpha}\cos(\theta)(b(\theta)\sin(\theta-\theta_c))'\frac{\chi'(r)}{r^2}\\
&+\frac{r_c}{\alpha}\left(
-(b(\theta)\sin(\theta-\theta_c))'\cos(\theta)
-\frac{1}{2}(b(\theta)\sin(\theta-\theta_c))''\sin(\theta)\right)
\frac{\chi(r)}{r^3}.
\end{align*}
Therefore we have
\begin{align*}
&\frac{1}{2}\partial_1\left(\frac{r_c}{\alpha}(b(\theta)\sin(\theta-\theta_c))'\frac{\chi(r)}{r^2}\right)\\
&-\partial_i\left(-\frac{r_c}{2\alpha}(b(\theta)\sin(\theta-\theta_c))'
\frac{\chi(r)}{r^2}M_\theta
-\frac{r_c}{\alpha}b(\theta)\sin(\theta-\theta_c)\frac{\chi(r)}{r^2}N_\theta\right)_{i1}\\
=&\frac{r_c}{\alpha}\left(\cos(\theta)(b(\theta)\sin(\theta-\theta_c))'-\sin(\theta)b(\theta)\sin(\theta-\theta_c) \right)\frac{\chi'(r)}{r^2}\\
=&\frac{r_c}{\alpha}\left(\cos(\theta)b(\theta)\sin(\theta-\theta_c)\right)'\frac{\chi'(r)}{r^2}.
\end{align*}
For $j=2$ we obtain
\begin{align*}
&\partial_i\left(-\frac{r_c}{2\alpha}(b(\theta)\sin(\theta-\theta_c))'
\frac{\chi(r)}{r^2}M_\theta
-\frac{r_c}{\alpha}b(\theta)\sin(\theta-\theta_c)\frac{\chi(r)}{r^2}N_\theta\right)_{i2}\\
=&-\frac{r_c}{2\alpha}(b(\theta)\sin(\theta-\theta_c))'(\cos(\theta)\sin(2\theta)-\sin(\theta)\cos(2\theta))\frac{\chi'(r)}{r^2}\\
&-\frac{r_c}{\alpha}b(\theta)\sin(\theta-\theta_c)(\cos(\theta)\cos(2\theta)+\sin(\theta)\sin(2\theta))\frac{\chi'(r)}{r^2}\\
&-\frac{r_c}{2\alpha}(b(\theta)\sin(\theta-\theta_c))''(-\sin(\theta)\sin(2\theta)-\cos(\theta)\cos(2\theta))\frac{\chi(r)}{r^3}\\
&-\frac{r_c}{\alpha}(b(\theta)\sin(\theta-\theta_c))'(-\sin(\theta)\cos(2\theta)+\cos(\theta)\sin(2\theta))\frac{\chi(r)}{r^3}\\
=&\frac{r_c}{\alpha}\left(-b(\theta)\sin(\theta-\theta_c)\cos(\theta)-\frac{1}{2}(b(\theta)\sin(\theta-\theta_c))'\sin(\theta)\right)\frac{\chi'(r)}{r^2}\\
&+\frac{r_c}{\alpha}\left(-(b(\theta)\sin(\theta-\theta_c))'\sin(\theta)+\frac{1}{2}(b(\theta)\sin(\theta-\theta_c))''\cos(\theta)\right)
\frac{\chi(r)}{r^3}.
\end{align*}
We now calculate the term involving $\tau$.
\begin{align*}
&\frac{1}{2}\partial_2\left(\frac{r_c}{\alpha}(b(\theta)\sin(\theta-\theta_c))'\frac{\chi(r)}{r^2}\right)\\
=&\frac{r_c}{2\alpha}\left(\frac{\chi'(r)}{r^2}-2\frac{\chi(r)}{r^3}\right)\sin(\theta)(b(\theta)\sin(\theta-\theta_c))'
+\frac{r_c}{2\alpha}\cos(\theta)(b(\theta)\sin(\theta-\theta_c))''\frac{\chi(r)}{r^3}\\
=&\frac{r_c}{2\alpha}\sin(\theta)(b(\theta)\sin(\theta-\theta_c))'\frac{\chi'(r)}{r^2}\\
&+\frac{r_c}{\alpha}\left(
-(b(\theta)\sin(\theta-\theta_c))'\sin(\theta)
+\frac{1}{2}(b(\theta)\sin(\theta-\theta_c))''\cos(\theta)\right)
\frac{\chi(r)}{r^3}.
\end{align*}
Therefore we have
\begin{align*}
&\frac{1}{2}\partial_2\left(\frac{r_c}{\alpha}(b(\theta)\sin(\theta-\theta_c))'\frac{\chi(r)}{r^2}\right)\\
&-\partial_i\left(-\frac{r_c}{2\alpha}(b(\theta)\sin(\theta-\theta_c))'
\frac{\chi(r)}{r^2}M_\theta
-\frac{r_c}{\alpha}b(\theta)\sin(\theta-\theta_c)\frac{\chi(r)}{r^2}N_\theta\right)_{i2}\\
=&\frac{r_c}{\alpha}\left(\sin(\theta)(b(\theta)\sin(\theta-\theta_c))'+\cos(\theta)b(\theta)\sin(\theta-\theta_c) \right)\frac{\chi'(r)}{r^2}\\
=&\frac{r_c}{\alpha}\left(\sin(\theta)b(\theta)\sin(\theta-\theta_c)\right)'\frac{\chi'(r)}{r^2}.
\end{align*}
In view of \eqref{chap2:defh2} and \eqref{chap2:deftau2}, this concludes the proof of Proposition \eqref{chap2:prph1}.

\subsection{Proof of Proposition \eqref{chap2:prph2}}\label{chap2:sech2}
Since $|\nabla \wht \lambda|\in H^1_{\delta+2}$, Lemma \eqref{chap2:produit2} implies that the terms of the form
$\frac{|b|}{1+r}|\nabla \wht \lambda|$ belong to $H^0_{\delta+3}$ and satisfy
$$\left\|\frac{|b|}{1+r}|\nabla \wht \lambda|\right\|_{H^0_{\delta+3}}
\lesssim \|b\|_{L^\infty(\m S^1)}\|\nabla \wht \lambda\|_{H^0_{\delta+2}}$$
and consequently, with the Sobolev injection $W^{1,2}(\m S^1)\subset L^\infty(\m S^1)$,
\begin{equation}
\label{chap2:est1}
\left\|\frac{|b|}{1+r}|\nabla \wht \lambda|\right\|_{H^0_{\delta+3}}\lesssim \|b\|_{W^{1,2}(\m S^1)}\|\wht \lambda \|_{H^1_{\delta+1}}.
\end{equation}
Moreover, thanks to Lemma \eqref{chap2:produit3}, the terms of the form
$\frac{r_c}{\alpha}\frac{|b|+|b'|}{(1+r)^2}|\nabla \wht \lambda|$ belong to $H^0_{\delta+3}$ and satisfy
\begin{equation}\label{chap2:est2}
\left\|\frac{r_c}{\alpha}\frac{|b|+|b'|}{(1+r)^2}|\nabla \wht \lambda|\right\|_{H^0_{\delta + 3}}\lesssim  \|b\|_{W^{1,2}(\m S^1)}\|\wht \lambda \|_{H^2_{\delta+1}},
\end{equation}
where we have used, that thanks to \eqref{chap2:aprio}, $\frac{r_c}{\alpha}\lesssim 1$.
The terms of the form $\frac{r_c}{\alpha}\frac{|b|+|b'|}{(1+r)^2}\frac{r_c}{(1+r)^2}$ are also in $H^0_{\delta+3}$ and satisfy
\begin{equation}
\label{chap2:est3}
\left\|\frac{r_c}{\alpha}\frac{|b|+|b'|}{(1+r)^2}\frac{r_c}{(1+r)^2}\right\|_{H^0_{\delta+3}}\lesssim \|b\|_{W^{1,2}(\m S^1)}|r_c|.
\end{equation}
Finally, since $\chi'$ is compactly supported, we also have
\begin{equation}
\label{chap2:est7}
\left\|\frac{|b|}{1+r}\alpha\chi'(r)\ln(r)\right\|_{H^0_{\delta+3}}\lesssim \|b\|_{W^{1,2}(\m S^1)}|\alpha|,\quad\left\|\frac{|b|}{1+r}\frac{r_c\chi'(r)}{(1+r)}\right\|_{H^0_{\delta+3}}\lesssim \|b\|_{W^{1,2}(\m S^1)}|r_c|, \quad...
\end{equation}
Consequently, the terms which remain to calculate are the ones decaying like $\frac{1}{r^2}$ and $\frac{1}{r^3}$. We obtain
\begin{align*}
H^{(2)}_{i1}\partial_i \lambda
=&-\frac{\chi(r)}{2r}b(\theta)\Bigg(-\left(\frac{\alpha}{r}+\frac{r_c\cos(\theta-\theta_c)}{r^2}\right)(\cos(2\theta)\cos(\theta)+\sin(2\theta)\sin(\theta))\\
&-\frac{r_c\sin(\theta-\theta_c)}{r^2}(-\sin(\theta)\cos(2\theta)+\cos(\theta)\sin(2\theta))\Bigg)
\\
&-\frac{r_c\chi(r)}{2\alpha r^2}(\cos(\theta-\theta_c)b(\theta)+\sin(\theta-\theta_c)b'(\theta))\frac{-\alpha}{r}(\cos(2\theta)\cos(\theta)+\sin(2\theta)\sin(\theta))\\
&-\frac{r_c\chi(r)}{\alpha r^2}\sin(\theta-\theta_c)b(\theta)\frac{-\alpha}{r}(-\cos(\theta)\sin(2\theta)+\sin(\theta)\cos(2\theta))+h_1\\
=&\frac{\alpha b(\theta)\chi(r)}{2r^2}\cos(\theta)+
\frac{r_cb(\theta)\chi(r)}{r^3}\left(\cos(\theta)\cos(\theta-\theta_c)-\frac{1}{2}\sin(\theta)\sin(\theta-\theta_c)\right)\\
&+\frac{r_cb'(\theta)\chi(r)}{2r^3}\cos(\theta)\sin(\theta-\theta_c)+h_1,
\end{align*}
where, thanks to \eqref{chap2:est1}, \eqref{chap2:est2} and \eqref{chap2:est3}, $h_1\in H^0_{\delta+3}$ satisfies
$$\|h_1\|_{H^0_{\delta+3}} \lesssim \|\lambda\|\|b\|_{W^{1,2}}.$$
We calculate
\begin{align*}
-\frac{1}{2}\tau^{(2)}\partial_1\lambda=&-\frac{b(\theta)\chi(r)}{2r}
\left(-\frac{\alpha}{r}\cos(\theta)-\frac{r_c\cos(\theta-\theta_c)}{r^2}\cos(\theta)+\frac{r_c\sin(\theta-\theta_c)}{r^2}\sin(\theta)\right)\\
&-\frac{r_c}{2\alpha}(b(\theta)\cos(\theta-\theta_c)+b'(\theta)\sin(\theta-\theta_c))\frac{\chi(r)}{r^2}\left(\frac{-\alpha}{r}\cos(\theta)\right)+h_2\\
=&\frac{\alpha b(\theta)\chi(r)}{2r^2}\cos(\theta)+
\frac{r_cb(\theta)\chi(r)}{r^3}\left(\cos(\theta)\cos(\theta-\theta_c)-\frac{1}{2}\sin(\theta)\sin(\theta-\theta_c)\right)\\
&+\frac{r_cb'(\theta)\chi(r)}{2r^3}\cos(\theta)\sin(\theta-\theta_c)+h_2,
\end{align*}
where thanks to \eqref{chap2:est1}, \eqref{chap2:est2} and \eqref{chap2:est3}, $h_2\in H^0_{\delta+3}$ satisfies
$$\|h_2\|_{h^0_{\delta+3}} \lesssim  \|\lambda\|\|b\|_{W^{1,2}}.$$
Therefore
$$-\frac{1}{2}\tau^{(2)}\partial_1\lambda-H^{(2)}_{i1}\partial_i \lambda=h_2-h_1\in H^0_{\delta+3}.$$
For $j=2$ we obtain
\begin{align*}
H^{(2)}_{i2}\partial_i \lambda
=&-\frac{\chi(r)}{2r}b(\theta)\Bigg(-\left(\frac{\alpha}{r}+\frac{r_c\cos(\theta-\theta_c)}{r^2}\right)(-\cos(2\theta)\sin(\theta)+\sin(2\theta)\cos(\theta))\\
&-\frac{r_c\sin(\theta-\theta_c)}{r^2}(-\cos(\theta)\cos(2\theta)-\sin(\theta)\sin(2\theta))\Bigg)
\\
&-\frac{r_c\chi(r)}{2\alpha r^2}(\cos(\theta-\theta_c)b(\theta)+\sin(\theta-\theta_c)b'(\theta)\frac{-\alpha}{r}(-\cos(2\theta)\sin(\theta)+\sin(2\theta)\cos(\theta))\\
&-\frac{r_c\chi(r)}{\alpha r^2}\sin(\theta-\theta_c)b(\theta)\frac{-\alpha}{r}(\sin(\theta)\sin(2\theta)+\cos(\theta)\cos(2\theta))+h_3\\
=&\frac{\alpha b(\theta)\chi(r)}{2r^2}\sin(\theta)+
\frac{r_cb(\theta)\chi(r)}{r^3}\left(\sin(\theta)\cos(\theta-\theta_c)+\frac{1}{2}\cos(\theta)\sin(\theta-\theta_c)\right)\\
&+\frac{r_cb'(\theta)\chi(r)}{2r^3}\sin(\theta)\sin(\theta-\theta_c)+h_3,
\end{align*}
where thanks to \eqref{chap2:est1}, \eqref{chap2:est2} and \eqref{chap2:est3}, $h_3\in H^0_{\delta+3}$ satisfies
$$\|h_3\|_{H^0_{\delta+3}} \lesssim \|\lambda\|\|b\|_{W^{1,2}}.$$
We calculate
\begin{align*}
-\frac{1}{2}\tau^{(2)}\partial_2\lambda=&-\frac{b(\theta)\chi(r)}{2r}
\left(-\frac{\alpha}{r}\sin(\theta)-\frac{r_c\cos(\theta-\theta_c)}{r^2}\sin(\theta)-\frac{r_c\sin(\theta-\theta_c)}{r^2}\cos(\theta)\right)\\
&-\frac{r_c}{2\alpha}(b(\theta)\cos(\theta-\theta_c)+b'(\theta)\sin(\theta-\theta_c))\frac{\chi(r)}{r^2}\left(\frac{-\alpha}{r}\sin(\theta)\right)+h_4\\
=&\frac{\alpha b(\theta)\chi(r)}{2r^2}\sin(\theta)+
\frac{r_cb(\theta)\chi(r)}{r^3}\left(\sin(\theta)\cos(\theta-\theta_c)+\frac{1}{2}\cos(\theta)\sin(\theta-\theta_c)\right)\\
&+\frac{r_cb'(\theta)\chi(r)}{2r^3}\sin(\theta)\sin(\theta-\theta_c)+h_4,
\end{align*}
where thanks to \eqref{chap2:est1}, \eqref{chap2:est2} and \eqref{chap2:est3}, $h_4\in H^0_{\delta+3}$ satisfies
$$\|h_4\|_{H^0_{\delta+3}} \lesssim \|\lambda\|\|b\|_{W^{1,\infty}}.$$
Therefore
$$-\frac{1}{2}\tau^{(2)}\partial_2\lambda-H^{(2)}_{i2}\partial_i \lambda=h_4-h_3\in H^0_{\delta+3}.$$
This concludes the proof of Proposition \eqref{chap2:prph2}.

\subsection{Proof of Proposition \eqref{chap2:prph3}}\label{chap2:sech3}
We first note
$$\partial_i(e^{-\lambda} H^{(3)})_{ij}+e^{-\lambda}H^{(3)}_{ij}\partial_j \lambda
=e^{-\lambda}\partial_i H^{(3)}_{ij},$$
$$\frac{1}{2}\partial_j \left(e^{-\lambda}\tau^{(3)}\right)-\frac{1}{2}e^{-\lambda}\tau^{(3)}\partial_j \lambda=\frac{1}{2}e^{-\lambda}\partial_j \tau^{(3)}-e^{-\lambda}\tau^{(3)}\partial_j \lambda.$$
Since $\wht \lambda\in H^2_{\delta+1}$, Proposition $\eqref{chap2:holder}$ implies that $\wht \lambda$ is bounded and consequently
$$|e^{-\lambda}|\lesssim (1+r^2)^{\frac{\alpha}{2}}.$$
Therefore Lemma \eqref{chap2:produit3} imply that the terms of the form
$e^{-\lambda}\frac{|B|+|B'|}{(1+r^2)}\nabla \wht \lambda$ belong to $H^0_{\delta+4-\alpha}$, with
$$\left\|e^{-\lambda}\frac{|B|+|B'|}{(1+r^2)}\nabla \wht \lambda\right\|_{H^0_{\delta+4-\alpha}} \lesssim \|B\|_{W^{1,2}(\m S^1)}\|\wht \lambda\|_{H^2_{\delta+1}}.$$
Since $\alpha$ is of size $\ep$, for $\ep$ small enough we have $\alpha<1$ and
\begin{equation}
\label{chap2:est4}
\left\|e^{-\lambda}\frac{|B|+|B'|}{(1+r^2)}\nabla \wht \lambda\right\|_{H^0_{\delta+3}} \lesssim \|B\|_{W^{1,2}(\m S^1)}\|\wht \lambda\|_{H^2_{\delta+1}}.
\end{equation}
The terms of the form $e^{-\lambda}\frac{|B|+|B'|}{1+r^2}\frac{r_c}{1+r^2}$ satisfy
$$\left|e^{-\lambda}\frac{|B|+|B'|}{1+r^2}\frac{r_c}{1+r^2}\right|\lesssim \frac{r_c(|B|+|B'|)}{(1+r^2)^{2-\frac{\alpha}{2}}},$$
so, for $\ep>0$ small enough so that $\delta+\alpha<0$ they belong to $H^0_{\delta+3}$ and satisfy
\begin{equation}
\label{chap2:est5}\left\|e^{-\lambda}\frac{|B|+|B'|}{1+r^2}\frac{r_c}{1+r^2}\right\|_{H^0_{\delta+3}}\lesssim \|B\|_{W^{1,2}(\m S^1)}|r_c|.
\end{equation}
Since $\chi'$ is smooth and compactly supported, the term of the form $e^{-\lambda}\frac{|B|+|B'|}{1+r^2}\chi'(r)$ belong to $H^0_{\delta+3}$ and satisfy
\begin{equation}
\label{chap2:est6}
\left\|e^{-\lambda}\frac{|B|+|B'|}{1+r^2}\chi'(r)\right\|_{H^0_{\delta+3}}
\lesssim \|B\|_{W^{1,2}(\m S^1)}.
\end{equation}
Consequently the terms which remain to calculate are the one which decay like $\frac{r^{\alpha}}{r^3}$.
We calculate
\begin{align*}
e^{-\lambda}\partial_i H^{(3)}_{i1}=&-e^{-\lambda}(1-\alpha)B'(\theta)(\sin(\theta)\sin(2\theta)+\cos(\theta)\cos(2\theta))\frac{\chi(r)}{r^3}\\
&-e^{-\lambda}\frac{B''(\theta)}{2}(-\sin(\theta)\cos(2\theta)+\cos(\theta)\sin(2\theta))\frac{\chi(r)}{r^3}+g_1\\
=&e^{-\lambda}\left(-(1-\alpha)B'(\theta)\cos(\theta)-\frac{1}{2}B''(\theta)\sin(\theta)\right)\frac{\chi(r)}{r^3}+g_1
\end{align*}
where we have used \eqref{calcmn} and where, thanks to the estimate \eqref{chap2:est6}, $g_1\in H^0_{\delta+3}$ satisfies
$$\|g_1\|_{H^0_{\delta+3}}\lesssim \|B\|_{W^{1,2}}.$$
We now calculate
\begin{align*}
&\frac{1}{2}e^{-\lambda}\partial_1 \tau^{(3)} 
-\tau^{(3)}e^{-\lambda}\partial_1 \lambda\\
&=-2\frac{1}{2}e^{-\lambda}B'(\theta)\frac{\chi(r)}{r^3}\cos(\theta)-\frac{1}{2}B''(\theta)\frac{e^{-\lambda}\chi(r)}{r^3}\sin(\theta)-B'(\theta)e^{-\lambda}\frac{\chi(r)}{r^2}\frac{-\alpha}{r}\cos(\theta)+g_2\\
&=(\alpha-1)e^{-\lambda}B'(\theta)\frac{\chi(r)}{r^3}\cos(\theta)
-\frac{1}{2}e^{-\lambda}B''(\theta)\frac{\chi(r)}{r^3}\sin(\theta)+g_2
\end{align*}
where thanks to the estimates \eqref{chap2:est4}, \eqref{chap2:est5} and \eqref{chap2:est6}, $g_2\in H^0_{\delta+3}$ satisfies
$$\|g_2\|_{H^0_{\delta+3}}\lesssim \|B\|_{W^{1,2}}.$$
Therefore 
$$\frac{1}{2}\partial_1 \tau^{(3)} 
-\frac{1}{2}\tau^{(3)}\partial_1 \lambda-\partial_i(e^{-\lambda} H^{(3)})_{1j}-e^{-\lambda}H^{(3)}_{1j}\partial_j \lambda=g_2-g_1\in H^0_{\delta+3}.$$
For $j=2$ we have
\begin{align*}
e^{-\lambda}\partial_i H^{(3)}_{i2}=&-e^{-\lambda}(1-\alpha)B'(\theta)(-\sin(\theta)\cos(2\theta)+\cos(\theta)\sin(2\theta))\frac{\chi(r)}{r^3}\\
&-e^{-\lambda}\frac{B''(\theta)}{2}(-\sin(\theta)\sin(2\theta)-\cos(\theta)\cos(2\theta))\frac{\chi(r)}{r^3}+g_3\\
=&e^{-\lambda}\left(-(1-\alpha)B'(\theta)\sin(\theta)+\frac{1}{2}B''(\theta)\cos(\theta)\right)\frac{\chi(r)}{r^3}+g_3
\end{align*}
where thanks to the estimate \eqref{chap2:est6}, $g_3\in H^0_{\delta+3}$ satisfies
$$\|g_3\|_{H^0_{\delta+3}}\lesssim \|B\|_{W^{1,2}}.$$
\begin{align*}
&\frac{1}{2}e^{-\lambda}\partial_2 \tau^{(3)} 
-e^{-\lambda}\tau^{(3)}\partial_2 \lambda\\
&=-2e^{-\lambda}\frac{1}{2}B'(\theta)\frac{\chi(r)}{r^3}\sin(\theta)+\frac{1}{2}B''(\theta)\frac{e^{-\lambda}\chi(r)}{r^3}\cos(\theta)-B'(\theta)e^{-\lambda}\frac{\chi(r)}{r^2}\frac{-\alpha}{r}\sin(\theta)+g_4\\
&=(\alpha-1)B'(\theta)\frac{\chi(r)e^{-\lambda}}{r^3}\sin(\theta)
+\frac{1}{2}B''(\theta)\frac{\chi(r)e^{-\lambda}}{r^3}\cos(\theta)+g_4
\end{align*}
where thanks to the estimates \eqref{chap2:est4}, \eqref{chap2:est5} and \eqref{chap2:est6}, $g_4\in H^0_{\delta+3}$ satisfies
$$\|g_4\|_{H^0_{\delta+3}}\lesssim \|B\|_{W^{1,2}}.$$
Therefore 
$$\frac{1}{2}\partial_2 \tau^{(3)} 
-\frac{1}{2}\tau^{(3)}\partial_2 \lambda-\partial_i(e^{-\lambda} H^{(3)})_{i2}-e^{-\lambda}H^{(3)}_{i2}\partial_i \lambda=g_4-g_3 \in H^0_{\delta+3}.$$
This conclude the proof of Proposition \eqref{chap2:prph3}.

\subsection{Proof of Proposition \eqref{chap2:prph4}}\label{chap2:sech4}
Recall that $f_j^{(1)}$ has been defined in \eqref{deffj}.
We calculate
\begin{equation}\label{chap2:int1}
\begin{split}
\int_{\m R^2} f^{(1)}_1 &=\int_{\m R^2} e^{\lambda}\left(-\dot{u}.\partial_1 u - A\Psi \partial_1\lambda +h^{(2)}_1+h^{(3)}_1\right) 
\\
&+\int (e^{\lambda}-1)\frac{\chi'(r)}{r}b(\theta)\cos(\theta)dx-\frac{r_c}{\alpha}\int e^{\lambda}\frac{\chi'(r)}{r^2}b(\theta)\sin(\theta-\theta_c)\cos(\theta)\partial_\theta \lambda\\
&+\pi \rho \cos(\eta),
\end{split}
\end{equation}
where we have used Proposition \ref{chap2:prph1} and the calculations
$$\frac{1}{2}\int e^{\lambda}A\partial_1 \Psi=-\frac{1}{2}\int e^{\lambda}A\Psi \partial_1\lambda,$$
\begin{align*}
\int e^{\lambda}\frac{\chi'(r)}{r}b(\theta)\cos(\theta)&=
\int (e^{\lambda}-1)\frac{\chi'(r)}{r}b(\theta)\cos(\theta)
+\left(\int \chi'(r)dr\right)\left(\int b(\theta)\cos(\theta)d\theta\right)\\
&=\int (e^{\lambda}-1)\frac{\chi'(r)}{r}b(\theta)\cos(\theta)+\pi \rho \cos(\eta),
\end{align*}
where we have used the definition of $b$ \eqref{chap2:defb} and the orthogonality condition \ref{ortob},
$$
\frac{r_c}{\alpha}\int e^{\lambda}\frac{\chi'(r)}{r^2}(b(\theta)\sin(\theta-\theta_c)\cos(\theta))'\\
=-\frac{r_c}{\alpha}\int e^{\lambda}\frac{\chi'(r)}{r^2}b(\theta)\sin(\theta-\theta_c)\cos(\theta)\partial_\theta \lambda .
$$
Similarly, we have
\begin{equation}\label{chap2:int2}\begin{split}
\int_{\m R^2} f^{(1)}_2 &=\int_{\m R^2} e^{\lambda}\left(-\dot{u}\partial_2 u - A\Psi \partial_2\lambda +h^{(2)}_2+h^{(3)}_2\right) \\
&+\int (e^{\lambda}-1)\frac{\chi'(r)}{r}b(\theta)\sin(\theta)dx-\frac{r_c}{\alpha}\int e^{\lambda}\frac{\chi'(r)}{r^2}b(\theta)\sin(\theta-\theta_c)\sin(\theta)\partial_\theta \lambda\\
&+\pi \rho \sin(\eta). 
\end{split}
\end{equation}

We calculate also
\begin{equation}\label{chap2:int3}\begin{split}
\int_{\m R^2} x_1f^{(1)}+x_2f^{(2)}=&\int_{\m R^2} e^{\lambda}\left(\dot{u}(r\partial_r u)
-A\Psi r\partial_r \lambda +x_1(h^{(2)}_1+h^{(3)}_1)+x_2(h^{(2)}_2+h^{(3)}_2)\right)\\
&+\int (e^{\lambda}-1)\chi'(r)b(\theta)
-\frac{r_c}{\alpha}\int e^{\lambda}\partial_\theta \lambda \frac{\chi'(r)}{r}b(\theta)\sin(\theta-\theta_c)\\
&+\left(\int \chi'(r)rdr\right)\int b(\theta)d\theta -A\int e^{\lambda}\Psi
\end{split}
\end{equation}
where we used $x_1\partial_1+x_2\partial_2= r\partial_r$ and the following calculations
$$\frac{1}{2}\int e^{\lambda}(x_1\partial_1 A\Psi +x_2 \partial_2 A \Psi)
=-\frac{1}{2}\int e^{\lambda}A\Psi(x_1\partial_1 \lambda+x_2 \partial_2 \lambda)-\int e^{\lambda}A\Psi,$$
\begin{align*}
&\int e^\lambda\left(x_1\frac{\chi'(r)}{r}b(\theta)\cos(\theta)+x_2\frac{\chi'(r)}{r}b(\theta)\sin(\theta)\right)\\
&=
\int e^{\lambda}\chi'(r)b(\theta)(\cos^2(\theta)+\sin^2(\theta))\\
&=
\left(\int \chi'(r)rdr\right)\left(\int b(\theta)d\theta\right)
+\int (e^{\lambda}-1)\chi'(r)b(\theta)
\end{align*}
\begin{align*}
&\int e^{\lambda}\frac{\chi'(r)}{r^2}\frac{r_c}{\alpha}\Bigg( x_1\left(b(\theta)\sin(\theta-\theta_c)\cos(\theta)\right)'
+x_2\left(b(\theta)\sin(\theta-\theta_c)\sin(\theta)\right)'\Bigg)\\
=&-\frac{r_c}{\alpha}\int e^{\lambda}\frac{\chi'(r)}{r}b(\theta)\sin(\theta-\theta_c)(-\cos(\theta)\sin(\theta)+\cos(\theta)\sin(\theta))\\
&-\frac{r_c}{\alpha}\int e^{\lambda}\partial_\theta \lambda \frac{\chi'(r)}{r}b(\theta)\sin(\theta-\theta_c)(\cos^2(\theta)+\sin^2(\theta))
\\
=& -\frac{r_c}{\alpha}\int e^{\lambda}\partial_\theta \lambda \frac{\chi'(r)}{r}b(\theta)\sin(\theta-\theta_c).
\end{align*}
Therefore, in view of \eqref{chap2:int1}, \eqref{chap2:int2} and \eqref{chap2:int3}, we have 
$$\int f^{(1)}_1 = \int f^{(1)}_2=\int x_1 f^{(1)}_1+x_2f^{(2)}_2=0$$
if and only if
the quantities $\rho \cos(\eta), \rho\sin(\eta)$ and $A$ are solutions of a linear system of the form
$$\left(\begin{array}{lll}1+O(\ep)&O(\ep)&O(\ep)\\
O(\ep)&1+O(\ep)&O(\ep)\\
O(\ep)&O(\ep)&1+O(\ep)
\end{array}\right)
\left( \begin{array}{l} \rho \cos(\eta)\\ \rho\sin(\eta) \\ A\end{array}\right)
=\left( \begin{array}{l} a_1\\ a_2 \\ a_3\end{array}\right),$$
where, since $\int \Psi=2\pi$,
\begin{align*}
a_1&=\frac{1}{\pi}\int \dot{u}\partial_1 u +O(\ep^2),\\
a_2&=\frac{1}{\pi}\int \dot{u}\partial_2 u +O(\ep^2),\\
a_3&=-\frac{1}{2\pi}\int \dot{u}r\partial_r u
+\frac{1}{2\pi}\left(\int \chi'(r)rdr\right)\int \wht b(\theta)d\theta + O(\ep^2).
\end{align*}
In the last equation we have used $\int b(\theta)d\theta = \int \wht b(\theta)d\theta$ to point out that this quantity does not depend on $\rho,\eta$.
For $\ep>0$ small enough, this system is invertible, therefore we can find a unique triplet $(\rho,\eta,A)$ in $\m R\times \m S^1\times \m R$ such that
the three integrals are zero, and we have
\begin{align*}
\rho \cos(\eta)&=\frac{1}{\pi}\int e^{\lambda}\dot{u}\partial_1 u +O(\ep^2),\\
\rho \sin(\eta)&=\frac{1}{\pi}\int e^{\lambda}\dot{u}\partial_2 u +O(\ep^2),\\
A&=-\frac{1}{2\pi}\int e^{\lambda}\dot{u}r\partial_r u
+\frac{C(\chi)}{2\pi}\int \wht b(\theta)d\theta + O(\ep^2).
\end{align*}
This concludes the proof of Proposition \eqref{chap2:prph4}

\section{The Lichnerowicz equation}\label{chap2:seclic}
Let $H$ and $\tau$ be given by
\begin{align*}
H&=e^{-\lambda}H^{(1)}+H^{(2)}+e^{-\lambda}H^{(3)},\\
\tau&= \tau^{(2)}+e^{-\lambda}\tau^{(3)}+A\Psi,
\end{align*}
with
$\rho,\eta,A$ and $H^{(1)}$ given by Proposition \eqref{chap2:prpeqh}. We recall
$H^{(1)}=J\frac{\chi(r)}{r^2}N_\theta +\wht H^{(1)},$ and
$$|A|+|J|+|\rho|+\|e^{-\lambda}\wht H^{(1)}\|_{\q H^1_{\delta+2}}\lesssim \ep.$$
\begin{prp}\label{chap2:prplamb}
 There exists a solution $\lambda'$ of \eqref{chap2:eqlamb} which can be written uniquely under the form
$$\lambda'=-\alpha'\chi(r)\ln(r)+r'_c\cos(\theta-\theta'_c)\frac{\chi(r)}{r}+\wht \lambda',$$
with $\wht \lambda'\in H^2_{\delta+1}$ and we have
\begin{align*}
\alpha'&=\frac{1}{4\pi}\int \left(\dot{u}^2+|\nabla u|^2\right) +O(\ep^2),\\
r'_c\cos(\theta'_c)&=\frac{1}{4\pi}\int x_1\left(\dot{u}^2+|\nabla u|^2\right) +O(\ep^2),\\
r'_c\sin(\theta'_c)&=\frac{1}{4\pi}\int x_2\left(\dot{u}^2+|\nabla u|^2\right) +O(\ep^2),
\end{align*}
and
$$\|\wht \lambda'\|_{H^2_{\delta+1}}
\lesssim 
\left\|\dot{u}^2 + |\nabla u|^2\right\|_{H^0_{\delta+3}}+
\ep^2.$$
\end{prp}
\begin{proof}
In order to apply Corollary \eqref{chap2:coro1} we have to check whether the right-hand side of \eqref{chap2:eqlamb} is in $H^0_{\delta+3}$.
To estimate $|e^{-\lambda}\wht H^{(1)}|^2$, we use Proposition \eqref{chap2:produit}, which yields
$|e^{-\lambda}\wht H^{(1)}|^2\in H^0_{\delta+3}$ with
\begin{equation}
\label{chap2:lest1}
\left\||e^{-\lambda}\wht H^{(1)}|^2\right\|_{ H^0_{\delta+3}}\lesssim 
\left\|e^{-\lambda}\wht H^{(1)}\right\|^2_{ H^1_{\delta+2}} \lesssim \ep^2.
\end{equation}
To estimate terms of the form $\frac{|b|}{1+r}e^{-\lambda}|\wht H^{(1)}|$ we use Lemma \eqref{chap2:produit2}. It yields
\begin{equation}
\label{chap2:lest2}\left\|\frac{|b|}{1+r}e^{-2\lambda}|\wht H^{(1)}|\right\|_{H^0_{\delta+3}}\lesssim \|b\|_{W^{1,2}(\m S^1)}\left\|e^{-\lambda}\wht H^{(1)}\right\|_{ H^1_{\delta+2}} \lesssim \ep^2.
\end{equation}
To estimate terms of the form $\frac{|B'|}{1+r^2}e^{-\lambda}|\wht H^{(1)}|$ and $\frac{r_c}{\alpha}\frac{|b'|}{1+r^2}e^{-\lambda}|\wht H^{(1)}|$ we use Lemma \eqref{chap2:produit3}, which yields
\begin{equation}
\label{chap2:lest3}
\begin{split}
\left\|\frac{|B|+|B'|}{1+r^2}e^{-\lambda}|\wht H^{(1)}|\right\|_{H^0_{\delta+3}}&\lesssim \|B\|_{W^{1,2}(\m S^1)}\left\|e^{-\lambda}\wht H^{(1)}\right\|_{ H^1_{\delta+2}} \lesssim \ep^2\\
\left\|\frac{r_c}{\alpha}\frac{|b|+|b'|}{1+r^2}e^{-\lambda}|\wht H^{(1)}|\right\|_{H^0_{\delta+3}}&\lesssim \|b\|_{W^{1,2}(\m S^1)}\left\|e^{-\lambda}\wht H^{(1)}\right\|_{ H^1_{\delta+2}} \lesssim \ep^2.
\end{split}
\end{equation}
In the same way, thanks to Proposition \ref{chap2:produit} and Lemma \ref{chap2:produit3} we estimate
\begin{equation}\label{chap2:lest4}
\|(A\Psi)^2\|_{H^0_{\delta+3}}\lesssim \ep^2, \quad \left\|\frac{|b|}{1+r}A\Psi\right\|_{H^0_{\delta+3}}\lesssim \ep^2, \quad ...
\end{equation}
We can also estimate
\begin{equation}
\label{chap2:lest5}
\left\|\frac{r_c^2}{\alpha^2}\frac{(|b|+|b'|)^2}{1+r^4}\right\|_{H^0_{\delta+3}}\lesssim \ep^2,\quad
\left\|e^{-2\lambda}\frac{(|B|+|B'|)^2}{1+r^4}\right\|_{H^0_{\delta+3}} \lesssim \ep^2, \quad ...
\end{equation}
We now calculate 
\begin{align*}
&\frac{1}{2}|H|^2-\frac{1}{4}\tau^2=
\frac{1}{2}\left(-b(\theta)\frac{\chi(r)}{2r}-\left(\frac{r_c}{2\alpha}(b(\theta)\sin(\theta-\theta_c))'+e^{-\lambda}\frac{B'(\theta)}{2}\right)\frac{\chi(r)}{r^2}\right)^2{M_\theta}^{ij}{M_{\theta}}_{ij}\\
&-b(\theta)\frac{\chi(r)}{2r}\left(-\frac{r_c}{\alpha}b(\theta)\sin(\theta-\theta_c)\frac{\chi(r)}{r^2}+e^{-\lambda}(J-(1-\alpha)B(\theta))\frac{\chi(r)}{r^2}\right){M_\theta}^{ij}{N_\theta}_{ij}\\
&-\frac{1}{4}\left(b(\theta)\frac{\chi(r)}{r}+\left(\frac{r_c}{\alpha}(b(\theta)\sin(\theta-\theta_c))'+e^{-\lambda}B'(\theta)\right)\frac{\chi(r)}{r^2}\right)^2+\wht h_1
\end{align*}
where thanks to the estimates \eqref{chap2:lest1}, \eqref{chap2:lest2}, \eqref{chap2:lest3}, \eqref{chap2:lest4}, and \eqref{chap2:lest5}, we have $\wht h_1 \in H^0_{\delta+3}$ with
$$\|\wht h_1 \|_{ H^0_{\delta+3}}\lesssim \ep^2.$$
Since ${M_\theta}^{ij}{M_{\theta}}_{ij}=2$ and ${M_\theta}^{ij}{N_\theta}_{ij}=0$ we obtain
$$\frac{1}{2}|H|^2-\frac{1}{4}\tau^2=\wht h_1\in H^0_{\delta+3}.$$
Consequently, we can solve \eqref{chap2:eqlamb} with Corollary \eqref{chap2:coro1}, and the solution $\lambda'$ can be written
$$\lambda'=-\alpha'\chi(r)\ln(r)+r_c\cos(\theta-\theta_c)\frac{\chi(r)}{r}+\wht \lambda',$$
with
\begin{align*}
\alpha'&= \frac{1}{2\pi}\int\left( \frac{1}{2}\dot{u}^2 + \frac{1}{2}|\nabla u|^2+\frac{1}{2}|H|^2-\frac{\tau^2}{4}\right)=\frac{1}{4\pi}\int \left(\dot{u}^2+|\nabla u|^2\right) +O(\ep^2),\\
r_c\cos(\theta_c)&=\frac{1}{2\pi}\int x_1\left( \frac{1}{2}\dot{u}^2 + \frac{1}{2}|\nabla u|^2+\frac{1}{2}|H|^2-\frac{\tau^2}{4}\right)=\frac{1}{4\pi}\int x_1\left(\dot{u}^2+|\nabla u|^2\right) +O(\ep^2),\\
r_c\sin(\theta_c)&=\frac{1}{2\pi}\int x_2\left( \frac{1}{2}\dot{u}^2 + \frac{1}{2}|\nabla u|^2+\frac{1}{2}|H|^2-\frac{\tau^2}{4}\right)=\frac{1}{4\pi}\int x_2\left(\dot{u}^2+|\nabla u|^2\right) +O(\ep^2),\\
\end{align*}
and $\wht \lambda' \in H^2_{\delta+1}$ such that
$$
\|\wht \lambda' \|_{H^2_{\delta+1}}\lesssim\left \|\frac{1}{2}\dot{u}^2 + \frac{1}{2}|\nabla u|^2+\frac{1}{2}|H|^2-\frac{\tau^2}{4}\right\|_{H^0_{\delta+3}}\lesssim 
\left\|\dot{u}^2 + |\nabla u|^2\right\|_{H^0_{\delta+3}}+
\ep^2.$$
This concludes the proof of Proposition \eqref{chap2:prplamb}.
\end{proof}

\section[Proof of the main theorem]{Proof of Theorem \eqref{chap2:main}}\label{chap2:secpf}
We find it more convenient to perform the fixed point with the quantities $(c_1,c_2)$
instead of $r_c,\theta_c$. We recall the relation
$$(c_1,c_2)=r_c(\cos(\theta_c),\sin(\theta_c)).$$
We note $X$ the Banach space
$$X=\m R\times \m R \times \m R\times H^2_{\delta+1}$$
equipped with the norm
$$\|\lambda\|_{X}=\|(\alpha, c_1, c_2, \wht \lambda)\|_{X}=|\alpha|+|c_1|+|c_2|+\|\wht \lambda\|_{H^2_{\delta+1}}.$$
We have constructed, for $\ep>0$ small enough, a map
$$F:X\rightarrow X$$
which maps $(\alpha, c_1, c_2, \wht \lambda)$ satisfying
$$\|(\alpha, c_1, c_2, \wht \lambda)\|_{X}=|\alpha|+|c_1|+|c_2|+\|\wht \lambda\|_{H^2_{\delta+1}}\lesssim \ep,$$
and $\alpha\geq \frac{1}{2}\alpha_{0}$ where
\begin{equation}
\label{chap2:conda}
\alpha_{0}=\frac{1}{4\pi}\left(\int \dot{u}^2+|\nabla u|^2\right),
\end{equation}
to $(\alpha', c'_1, c'_2, \wht \lambda')$ such that, for $\rho,\eta, A,  H^{(1)}$  given by Proposition \eqref{chap2:prpeqh}, if we note
\begin{align*}
\lambda&=-\alpha \chi(r)\ln(r)+ r_c\cos(\theta-\theta_c)\frac{\chi(r)}{r} +\wht \lambda,\\
H&= e^{-\lambda}H^{(1)}+H^{(2)}+e^{-\lambda}H^{(3)},\\
\tau&= \tau^{(2)}+e^{-\lambda}\tau^{(3)}+A\Psi,
\end{align*}
then $H$ satisfies
$$\partial_i H_{ij} +H_{ij}\partial_i \lambda  = -\dot{u}.\partial_j u + \frac{1}{2} \partial_j \tau-\frac{1}{2} \tau\partial_j \lambda,$$
and 
$$\lambda'=-\alpha' \chi(r)\ln(r)+ r'_c\cos(\theta-\theta'_c)\frac{\chi(r)}{r} +\wht \lambda'$$
is the solution of
$$\Delta \lambda' + \frac{1}{2}\dot{u}^2 + \frac{1}{2}|\nabla u|^2+\frac{1}{2}|H|^2-\frac{\tau^2}{4} = 0,$$
given by Proposition \eqref{chap2:prplamb}.
Proposition \eqref{chap2:prpeqh} implies
$$|\rho|+ |J|+|A|+\|\wht H^{(1)}\|_{\q H^1_{\delta+2+\alpha}} \lesssim \ep,$$
and Proposition \eqref{chap2:prplamb} implies
$$|r'_c|+|\alpha'|+\|\wht \lambda'\|_{H^2_{\delta+1}}\lesssim \ep.$$
In particular there exist $C_0$ such that
$$\|F(\alpha_0,0,0,0)\|_{X}=C_0\ep.$$

Next we show that $F$ is a contracting map in
$$B_X(0,2C_0 \ep)\cap\left\{ \alpha \geq \frac{\alpha_0}{2}\right\}.$$
We consider, for $i=1,2$ $(\alpha_i,(c_1)_i,(c_2)_i,\wht \lambda_i)$ such that
$$\|(\alpha_i,(c_1)_i,(c_2)_i,\wht \lambda_i)\|_{X}\leq 2C_0\ep, \quad \alpha_i \geq \frac{\alpha_0}{2}.$$
We note
$$(\alpha'_i,(c'_1)_i,(c'_2)_i,\wht \lambda'_i)=F(\alpha_i,(c_1)_i,(c_2)_i,\wht \lambda_i),$$
$$(r_c)_i(\cos(\theta_c)_i,\sin(\theta_c)_i)=((c_1)_i,(c_2)_i), \quad
(r'_c)_i(\cos(\theta'_c)_i,\sin(\theta'_c)_i)=((c'_1)_i,(c'_2)_i).$$
Since $\alpha'_i= \alpha_0 +O(\ep^2)$ we have for $\ep$ small enough
$$\alpha'_i \geq \frac{\alpha_0}{2}.$$
We note $\rho_i,\eta_i,A_i, J_i, \wht H^{(1)}_i$ the corresponding quantities given by Proposition \eqref{chap2:prpeqh}. The proof of the following lemma is postponed to the end of this section.

\begin{lm}\label{chap2:lml}
We have the estimate
$$|\alpha'_1-\alpha'_2|+|(c'_1)_1-(c'_1)_2|+ |(c'_2)_1-(c'_2)_2|+\|\wht  \lambda'_1-\wht \lambda'_2\|_{H^2_{\delta+1}}\lesssim \ep \|\lambda_1-\lambda_2\|_X.$$
\end{lm}
We are now in position to prove Theorem \eqref{chap2:main}. Thanks to Lemma \eqref{chap2:lml} there exists $C$ such that
$$\|F(\lambda_1)-F(\lambda_2)\|_X \leq C\ep \|\lambda_1-\lambda_2\|_X.$$
Consequently, by taking $\lambda_2=(\alpha_0,0,0,0)$ we have 
$$\forall \lambda \in B_X(0,2C_0 \ep)\cap\left\{ \alpha \geq \frac{\alpha_0}{2}\right\}, \quad
\|F(\lambda)-F(\alpha_0,0,0,0)\|\leq 2CC_0\ep^2.$$
Therefore, if $\ep$ is small enough such that $C\ep \leq 1$, the map $F$ sends
$B_X(0,2C_0 \ep)$ into itself. Moreover we already have noted that the condition $\alpha \geq \frac{\alpha_0}{2}$ is preserved by $F$ for $\ep$ small enough. Finally, for $C\ep<1$ the map $F$ is contracting, and the Picard fixed point Theorem yields the existence of a fixed point.

We now choose coordinates centered in the center of mass $(c_1,c_2)$. 
For these coordinates, we have $r_c=0$ and consequently
\begin{align*}
\lambda&= -\alpha \chi(r)\ln(r)+\wht \lambda,\\
H&=-(\wht b(\theta)+\rho\cos(\theta-\eta))\frac{\chi(r)}{2r}M_{\theta} + e^{-\lambda}\frac{\chi(r)}{(r)^2}\left((J-(1-\alpha)B(\theta))N_{\theta} -\frac{B'(\theta)}{2}M_{\theta} \right) + \wht H,\\
\tau&=(\wht b(\theta)+\rho \cos(\theta-\eta))\frac{\chi(r)}{r}+e^{-\lambda} B'(\theta)\frac{\chi(r)}{r^2}+A\Psi,
\end{align*}
The estimates of Propositions \eqref{chap2:prpeqh} and \eqref{chap2:prplamb} 
complete the proof of Theorem \eqref{chap2:main}.

To prove Lemma \ref{chap2:lml}, we first prove the following two lemmas.
\begin{lm}\label{chap2:lmrho}
We have the estimate
$$|\rho_1 \cos(\eta_1)-\rho_2\cos(\eta_2)|
+|\rho_1 \sin(\eta_1)-\rho_2\sin(\eta_2)|+|A_1-A_2|\lesssim \ep\|
\lambda_1-\lambda_2\|_{X}.$$
\end{lm}
\begin{lm}\label{chap2:estH}
We have the estimate
$$\|e^{-\lambda_1}\wht H^{(1)}_1-e^{-\lambda_2}\wht H^{(1)}_1\|_{H^1_{\delta+2}}
+|J_1-J_2|\lesssim \ep \|
\lambda_1-\lambda_2\|_{X}.$$
\end{lm}

\begin{proof}[Proof of Lemma \eqref{chap2:lmrho}]
The quantities $\rho_i\cos(\theta_i), \rho_i\sin(\theta_i), A_i$ are given by the expressions \eqref{chap2:int1}, \eqref{chap2:int2}, \eqref{chap2:int3}. Therefore we have
\begin{equation}\label{chap2:express}\begin{split}
&\pi(\rho_1\cos(\eta_1)-\rho_2\cos(\eta_2))\\=&\int_{\m R^2} \left(e^{\lambda_1}-e^{\lambda_2}\right)\dot{u}\partial_1 u + e^{\lambda_1}A_1\Psi \partial_1\lambda_1-e^{\lambda_2}A_2\Psi \partial_1\lambda_2\\
&-\int e^{\lambda_1}(h^{(2)}_1)_1-e^{\lambda_2}(h^{(2)}_1)_2+e^{\lambda_1}(h^{(3)}_1)_1-e^{\lambda_2}(h^{(3)}_1)_2\\
&-\int (e^{\lambda_1}-e^{\lambda_2})\frac{\chi'(r)}{r}b_1(\theta)\cos(\theta)dx
+(e^{\lambda_2}-1)\frac{\chi'(r)}{r}(\rho_1\cos(\theta-\eta_1)-\rho_2\cos(\theta-\eta_2))\\
&+\int \frac{(r_c)_1}{\alpha_1}
 e^{\lambda_1}\frac{\chi'(r)}{r}b_1(\theta)\sin(\theta-(\theta_c)_1)\cos(\theta)\partial_\theta \lambda_1-\frac{(r_c)_2}{\alpha_2}
  e^{\lambda_2}\frac{\chi'(r)}{r}b_2(\theta)\sin(\theta-(\theta_c)_2)\cos(\theta)\partial_\theta \lambda_2,
\end{split}
\end{equation}
and a similar expression for $\rho_1\sin(\eta_1)-\rho_2\sin(\eta_2)$
and $A_1-A_2$.

We estimate first 
$(h^{(2)}_j)_1-(h^{(2)}_j)_2,$
where the quantities $(h^{(2)}_j)_i$ are defined by \eqref{chap2:defhj2}. We have
$$(h^{(2)}_j)_1-(h^{(2)}_j)_2=-\frac{1}{2}\tau^{(2)}_1\partial_j(\lambda_1-\lambda_2)+\frac{1}{2}(\tau^{(2)}_1-\tau^{(2)}_2)\partial_j \lambda_2
-(H_{ij}^{(2)})_1\partial_i (\lambda_1-\lambda_2)
+\left((H_{ij}^{(2)})_1-(H_{ij}^{(2)})_2\right)\partial_i \lambda_2.$$
We calculate
\begin{align*}
&\tau^{(2)}_1-\tau^{(2)}_2=(\rho_1\cos(\theta-\theta_1)-\rho_2\cos(\theta-\theta_2))\\+&\left(\frac{(r_c)_1}{\alpha_1}((\rho_1\cos(\theta-\eta_1)+\wht b(\theta))\cos(\theta-(\theta_c)_1))'-\frac{(r_c)_2}{\alpha_2}((\rho_2\cos(\theta-\eta_2)+\wht b(\theta))\cos(\theta-(\theta_c)_2))'\right)\frac{\chi(r)}{r^2}.
\end{align*}
We have a similar expression for $(H_{ij}^{(2)})_1-(H_{ij}^{(2)})_2$.
Therefore we have
\begin{equation}
\label{chap2:esthh}
\left\|(h^{(2)}_j)_1-(h^{(2)}_j)_2\right\|_{H^0_{\delta+3}}\lesssim \ep|\rho_1\cos(\eta_1)-\rho_2\cos(\eta_2)|+\ep|\rho_1\sin(\eta_1)-\rho_2\sin(\eta_2)|
+\ep\|
\lambda_1-\lambda_2\|_{X}.
\end{equation}
We now estimate $(h^{(3)}_j)_1-(h^{(3)}_j)_2$, where the quantities $(h^{(3)}_j)_i$ are defined by \eqref{chap2:defgj2}. The function $\tau^{(3)}$ does not depend on the index $i=1,2$. We calculate
$$H^{(3)}_1-H^{(3)}_2=\frac{\chi(r)}{r^2}(\alpha_1-\alpha_2)B(\theta)N_\theta.$$
Therefore we obtain
\begin{equation}
\label{chap2:esthh3}
\left\|(h^{(3)}_j)_1-(h^{(3)}_j)_2\right\|_{H^0_{\delta+3}}\lesssim 
\ep\|
\lambda_1-\lambda_2\|_{X}.
\end{equation}
The estimates for the other terms of \eqref{chap2:express} are similar. Therefore \eqref{chap2:express}, together with the estimates \eqref{chap2:esthh} and \eqref{chap2:esthh3} yields
$$| \rho_1\cos(\eta_1)-\rho_2\cos(\eta_2)|\lesssim \ep\left( |\rho_1\cos(\eta_1)-\rho_2\cos(\eta_2)|+|\rho_1\sin(\eta_1)-\rho_2\sin(\eta_2)|+|A_1-A_2|\right) + \ep \|\lambda_1-\lambda_2\|_X.$$
Similarly we obtain
$$| \rho_1\sin(\eta_1)-\rho_2\sin(\eta_2)|\lesssim \ep\left( |\rho_1\cos(\eta_1)-\rho_2\cos(\eta_2)|+|\rho_1\sin(\eta_1)-\rho_2\sin(\eta_2)|+|A_1-A_2|\right) + \ep \|\lambda_1-\lambda_2\|_X$$
$$| A_1-A_2|\lesssim \ep\left( |\rho_1\cos(\eta_1)-\rho_2\cos(\eta_2)|+|\rho_1\sin(\eta_1)-\rho_2\sin(\eta_2)|+|A_1-A_2|\right) + \ep \|\lambda_1-\lambda_2\|_X$$
and consequently
$$|\rho_1 \cos(\theta_1)-\rho_2\cos(\theta_2)|
+|\rho_1 \sin(\theta_1)-\rho_2\sin(\theta_2)|+|A_1-A_2|\lesssim \ep\|
\lambda_1-\lambda_2\|_{X},$$
which concludes the proof of Lemma \eqref{chap2:lmrho}.
\end{proof}

\begin{proof}[Proof of Lemma \eqref{chap2:estH}]
We compare first $J_1$ and $J_2$ thanks to the formula \eqref{chap2:defJ}.
We obtain
\begin{align*}
J_1-J_2 &=\frac{1}{2\pi}
\int_{\m R^2} -\left(e^{\lambda_1}-e^{\lambda_2}\right)\dot{u}\partial_\theta u
-(e^{\lambda_1}A_1-e^{\lambda_2}A_2)\Psi \partial_\theta \lambda_1-A_2e^{\lambda_2}\Psi \partial_\theta(\lambda_1-\lambda_2) \\ 
&+\frac{\rho_1 (r_c)_1}{\alpha_1}\sin(\eta-(\theta_c)_1)-\frac{\rho_2 (r_c)_2}{\alpha_2}\sin(\eta-(\theta_c)_2)+s.t.
\end{align*}
where the notation $s.t.$ stands for similar terms.
Therefore,  we obtain
$$|J_1-J_2|\lesssim \ep\|\lambda_1-\lambda_2\|_X + 
|\rho_1\cos(\eta_1)-\rho_2\cos(\eta_2)|+|\rho_1\sin(\eta_1)-\rho_2\sin(\eta_2)|+|A_1-A_2|$$
and thanks to Lemma \eqref{chap2:lmrho} we infer
\begin{equation}
\label{chap2:estJJ}
|J_1-J_2|\lesssim \ep\|\lambda_1-\lambda_2\|_X.
\end{equation}
We now write the equation satisfied by $e^{-\lambda_1}\wht H^{(1)}_1-e^{-\lambda_2}\wht H^{(1)}_2$
\begin{align*}
&\partial_i \left(e^{-\lambda_1}\wht H^{(1)}_1-e^{-\lambda_2}\wht H^{(1)}_2\right)_{ij}\\
= & e^{-\lambda_1}(\wht H^{(1)}_1)_{ij}\partial_j \lambda_1-e^{-\lambda_2}(\wht H^{(1)}_2)_{ij}\partial_j \lambda_2\\
& + (e^{-\lambda_1}J_1-e^{\lambda_2}J_2)\partial_i\left(\frac{\chi(r)}{r^2}N_\theta\right)+ e^{\lambda_1}\partial_i(H^{(1)}_1)_{ij}- e^{\lambda_2}\partial_i(H^{(1)}_2)_{ij}\\
=&\left(e^{-\lambda_1}(\wht H^{(1)}_1)_{ij}-e^{-\lambda_2}(\wht H^{(1)}_2)_{ij}\right)\partial_j \lambda_1+e^{-\lambda_2}(\wht H^{(1)}_2)_{ij}\partial_j(\lambda_1-\lambda_2)\\
&+(e^{-\lambda_1}J_1-e^{\lambda_2}J_2)\partial_i\left(\frac{\chi(r)}{r^2}N_\theta\right)+(A_1-A_2)\partial_j \Psi+(h^{(2)}_j)_1-(h^{(2)}_j)_2+s.t.
\end{align*}
 Consequently, Corollary \eqref{chap2:coro2} yields
$$\|e^{-\lambda_1}\wht H^{(1)}_1-e^{-\lambda_2}\wht H^{(1)}_2\|_{\q H^1_{\delta+2}}
\lesssim \ep \|e^{-\lambda_1}\wht H^{(1)}_1-e^{-\lambda_2}\wht H^{(1)}_2\|_{\q H^1_{\delta+2}}+|J_1-J_2|+ \ep\|\lambda_1-\lambda_2\|_X,$$
and thanks to \eqref{chap2:estJJ}
$$\|e^{-\lambda_1}\wht H^{(1)}_1-e^{-\lambda_2}\wht H^{(1)}_2\|_{\q H^1_{\delta+2}}
\lesssim  \ep\|\lambda_1-\lambda_2\|_X,$$
which concludes the proof of Lemma \eqref{chap2:estH}.
\end{proof}

\begin{proof}[Proof of Lemma \eqref{chap2:lml}]
In view of \eqref{chap2:eqlamb} we have
$$\Delta(\lambda'_1-\lambda'_2)=-\frac{1}{2}|H_1|^2+\frac{1}{4}\tau_1^2
+\frac{1}{2}|H_2|^2-\frac{1}{4}\tau_2^2.$$
The right-hand side is in $H^0_{\delta+3}$ and satisfies
\begin{align*}
&\left\|\frac{1}{2}|H_1|^2-\frac{1}{4}\tau_1^2
-\frac{1}{2}|H_2|^2-\frac{1}{4}\tau_2^2\right\|_{H^0_{\delta+3}}\\
&\lesssim \ep\left(\|e^{-\lambda_1}\wht H^{(1)}_1-e^{-\lambda_2}\wht H^{(1)}_2\|_{H^1_{\delta+2}}+|J_1-J_2|\right)+\ep \|\lambda_1-\lambda_2\|_X\\
&\lesssim \ep \|\lambda_1-\lambda_2\|_X,
\end{align*}
where we have used Lemma \ref{chap2:estH} in the last inequality.
Therefore Corollary \eqref{chap2:coro1} allows us to write
$$\lambda'_1-\lambda'_2=-(\alpha'_1-\alpha'_2)\chi(r)\ln(r)
 +((r'_c)_1\cos(\theta-(\theta'_c)_1)-(r'_c)_2\cos(\theta-(\theta'_c)_2))\frac{\chi(r)}{r} +\wht \lambda'_1-\wht \lambda'_2,$$
 with
 $$|\alpha'_1-\alpha'_2|+|(c'_1)_1-(c'_1)_2|+ |(c'_2)_1-(c'_2)_2|+\|\wht  \lambda'_1-\wht \lambda'_2\|_{H^2_{\delta+1}}\lesssim \ep \|\lambda_1-\lambda_2\|_X.$$
 This concludes the proof of Lemma \eqref{chap2:lml}.
\end{proof}

\bibliographystyle{plain}
\bibliography{contrainte}

\end{document}